\documentclass[11pt,reqno,letterpaper]{amsart}
\usepackage{amssymb}
\usepackage{graphicx,color}
\usepackage{hyperref}
\usepackage[normalem]{ulem}
\usepackage{fullpage} 
 
\usepackage{scalerel}[2014/03/10]
\usepackage[usestackEOL]{stackengine}

\def\dashint{\,\ThisStyle{\ensurestackMath{%
  \stackinset{c}{.2\LMpt}{c}{.5\LMpt}{\SavedStyle-}{\SavedStyle\phantom{\int}}}%
  \setbox0=\hbox{$\SavedStyle\int\,$}\kern-\wd0}\int}
\def\ddashint{\,\ThisStyle{\ensurestackMath{%
  \stackinset{c}{.2\LMpt}{c}{.5\LMpt+.2\LMex}{\SavedStyle-}{%
    \stackinset{c}{.2\LMpt}{c}{.5\LMpt-.2\LMex}{\SavedStyle-}{%
      \SavedStyle\phantom{\int}}}}\setbox0=\hbox{$\SavedStyle\int\,$}\kern-\wd0}\int}

\theoremstyle{plain}

\begin{document}


\theoremstyle{plain}
\newtheorem{theorem}{Theorem} [section]
\newtheorem{corollary}[theorem]{Corollary}
\newtheorem{lemma}[theorem]{Lemma}
\newtheorem{proposition}[theorem]{Proposition}
\newtheorem{example}[theorem]{Example}


\theoremstyle{definition}
\newtheorem{definition}[theorem]{Definition}
\theoremstyle{remark}
\newtheorem{remark}[theorem]{Remark}

\def\RR{{\mathbb R}}
\def\NN{{\mathbb N}}
\def\N{{\mathbb N}}
\def\LE{\mathcal{E}}
\def\gabo{\textcolor{red}}
\def\fran{\textcolor{magenta}}
\def\julio{\textcolor{blue}}

\numberwithin{theorem}{section}
\numberwithin{equation}{section}

\title[Coupling local and nonlocal equations]{Coupling local and nonlocal equations with Neumann boundary conditions.}

\thanks{G.A. partially supported by ANPCyT under grant PICT 2018 - 3017  (Argentina). \newline
\indent F.M.B. partially supported by ANPCyT under grant PICT 2018 - 3017 (Argentina).
 \newline
\indent J.D.R. partially supported by 
CONICET grant PIP GI No 11220150100036CO
(Argentina), PICT 2018 - 3183 (Argentina) and UBACyT grant 20020160100155BA (Argentina).}

\author[G. Acosta, F. Bersetche, J.D. Rossi]{Gabriel Acosta, Francisco Bersetche, Julio D. Rossi} 
\address{Departamento  de Matem\'atica, FCEyN, Universidad de Buenos Aires \&
IMAS CONICET,  Pabell\'on I, Ciudad Universitaria (1428),
Buenos Aires, Argentina.}
\email{gacosta@dm.uba.ar, fbersetche@dm.uba.ar, jrossi@dm.uba.ar}

%
%

\date{}

\begin{abstract}
We introduce two different ways of coupling local and nonlocal equations
with Neumann boundary conditions in such a 
way that the resulting model is naturally associated with an energy functional. 
For these two models we prove that there is a minimizer of the resulting energy that
is unique modulo adding a constant. 
\end{abstract}

\maketitle


\section{Introduction} 

Nonlocal models can be used to describe phenomena 
(including
problems characterized by long-range interactions and discontinuities)
that can not be well represented by classical Partial Differential Equations, PDE. 
For instance, long-range interactions effectively
describe anomalous diffusion and crack
formation results in material models. 
The fundamental difference between nonlocal models and classical local models  
is the fact that the latter only involve differential operators (local equations), whereas the
former rely on integral operators (nonlocal equations). For general references on nonlocal models 
and their applications to elasticity, population dynamics, 
image processing, etc, we refer to
to \cite{BCh,Bere,CaRo,CF,ChChRo,CERW,Cortazar-Elgueta-Rossi-Wolanski,DiPaola,F,Hutson,MenDu,Sil,Sil1,
Sil2,Strick,W,Z} and the book \cite{ElLibro}.

It is often the case that nonlocal effects occur only in
some parts of the domain, whereas, in the remaining parts, the system can be accurately
described by a local equation. The goal of coupling local and nonlocal models is to
combine a local equation (a PDE) with a nonlocal one (an integral equation) acting in different
parts of the domain,
under the assumption that the spacial location of local and nonlocal effects can be identified
in advance. 
In this
context, one of the challenges of a coupling strategy is to provide a mathematically 
consistent formulation.

Along this work we consider an open bounded domain $\Omega \subset \mathbb{R}^N$. 
In our models it is assumed that $\Omega$ is 
divided into two disjoint subdomains; a local region that we will denote by $\Omega_\ell$ and
a nonlocal region, $\Omega_{n\ell}$. Thus we have $\Omega_\ell,\Omega_{n\ell}\subset \Omega\subset \mathbb{R}^N$
with 
$\Omega = (\overline{\Omega}_\ell \cup \overline{\Omega}_{n\ell})^{\circ}$.
 Our main goal in this paper is to introduce two different ways of coupling a local
 classical PDE in $\Omega_{\ell}$ with a nonlocal equation in $\Omega_{n\ell}$
 in such a way that the resulting problem is naturally associated with an energy
 functional that is invariant under the addition of a constant (as is the usual case for
 Neumann boundary conditions in the literature). This paper is a
 continuation of \cite{ABR} where we tackled the Dirichlet case. 
 
 Let us first recall some well-known facts: for the classical Laplacian (a local operator) with homogeneous 
 Neumann boundary conditions the model problem read as
 \begin{equation}
\label{eq:main.Neumann.local.Lapla}
\begin{cases}
\displaystyle \Delta u (x) = f(x), \ &x\in \Omega_\ell,
\\[5pt]
\displaystyle \frac{\partial u}{\partial \eta} (x)=0,\qquad & x\in \partial \Omega_\ell.
\end{cases}
\end{equation}
Here, $u$ is the unknown and $f$ is an external source. 
Associated to the problem we have the natural energy 
\begin{equation}
\label{energy.Lapla.intro}
F_{\ell}(u) = \frac12 \int_{\Omega_\ell} |\nabla u (x)|^2 \, dx - \int_{\Omega_\ell} f(x) u(x) \, dx
\end{equation}
that is well posed in the space $H_\ell=\{u \in H^1 (\Omega_\ell)  :  \int_{\Omega_\ell} u (y) \, dy =0 \}$.  
Notice that we need to impose that $\int_{\Omega_{\ell}} f(x) \, dx=0$ in order to have existence 
of solutions to \eqref{eq:main.Neumann.local.Lapla} 
and in this case we get existence and uniqueness for \eqref{eq:main.Neumann.local.Lapla}
modulo an additive constant: there is a unique solution to \eqref{eq:main.Neumann.local.Lapla}
with $\int_{\Omega_{\ell}} u(x) \, dx=0$ (that is obtained as a minimizer of the energy $F_{\ell}$ in $H_\ell$)
and any other solution can be obtained by adding a constant. 
For the proofs we refer to the textbook \cite{evans}.

For a nonlocal counterpart of \eqref{eq:main.Neumann.local.Lapla}
in $\Omega_{n\ell}$,
we need to introduce a nonnegative kernel $J:\RR^N \mapsto \RR$ and then we 
consider as a nonlocal analogous to \eqref{eq:main.Neumann.local.Lapla} the following equation,
\begin{equation}
\label{eq:main.Neumann.nonlocal.intro}
\displaystyle
2 \int_{\Omega_{n\ell}} J(x-y)(u(y)-u(x))\, dy = f(x),
\quad x \in \Omega_{n\ell}.
\end{equation}
Assuming that the kernel is symmetric, $J(z) = J(-z)$, we have the associated energy functional
\begin{equation}
\label{energy.no-loc.intro}
F_{n\ell}(u) = \frac12 \int_{\Omega_{n\ell}} \int_{\Omega_{n\ell}} J(x-y)|u(y)-u(x)|^2\, dy \, dx - 
\int_{\Omega_{n\ell}} f(x) u(x) \, dx.
\end{equation}
When the kernel is in $L^1$ this functional can be
considered in the space $H_{n\ell} = \{ u \in L^2 (\Omega_{n\ell})  : \int_{\Omega_{n\ell}} u(x) \, dx=0 \}$. 
Like for the local case, we need to assume that $\int_{\Omega_{n\ell}} f(x) \, dx=0$ and 
again we have existence and uniqueness for solutions to \eqref{eq:main.Neumann.nonlocal.intro}
modulo a constant (there is a unique solution to \eqref{eq:main.Neumann.local.Lapla}
with $\int_{\Omega_{n\ell}} u(x) \, dx=0$ and it is obtained as a minimizer of $F_{n\ell}$ in $H_{n\ell}$), see \cite{ElLibro}. 

Our main goal here is to present two different ways of coupling local (in $\Omega_\ell$) and nonlocal models
(in $\Omega_{n\ell}$ in 
in such a way that the resulting problem (in the whole $\Omega$) has the same properties as the 
previous two Neumann problems; the problem is naturally associated with an 
energy functional and it has a unique solution (up to an additive constant) when
the external source is such that $\int_\Omega f(x)\, dx =0$.

{\subsection{First model. Volumetric couplings.} Let us present our first model coupling local and nonlocal equations in 
two disjoint subdomains $\Omega_\ell,\Omega_{n\ell}\subset \Omega\subset \mathbb{R}^N$.
For
$u: \Omega \mapsto \mathbb{R}$,
we consider the local/nonlocal energy
\begin{equation} \label{Jyvaskyla}
\begin{array}{l}
\displaystyle 
E_{Neu}(u):=\int_{\Omega_\ell} \frac{|\nabla u (x)|^2}{2} \, dx 
+ \frac{1}{2}\int_{\Omega_{n\ell}}\int_{\Omega_{n\ell}} J(x-y)|u(y)-u(x)|^2\, dy \, dx
\\[10pt]
\displaystyle \qquad\qquad  \qquad
+ \frac{1}{2}\int_{\Omega_{n\ell}}\int_{\Omega_\ell} G(x,y)|u(y)-u(x)|^2\, dy\,  dx
- \int_\Omega f (x) u (x) \, dx,
\end{array}
\end{equation}
and we look for critical points (minimizers) of this energy and  the corresponding 
equations that they satisfy. 

In the functional \eqref{Jyvaskyla} we can identify the local part of the energy in $\Omega_\ell$,
\begin{equation} \label{local.part}
\int_{\Omega_\ell} \frac{|\nabla u(x)|^2}{2}\, dx,
\end{equation}
the nonlocal part, acting in $\Omega_{n\ell}$, 
\begin{equation} \label{nonlocal.part}
 \frac{1}{2}\int_{\Omega_{n\ell}}\int_{\Omega_{n\ell}} J(x-y)|u(y)-u(x)|^2\, dy \, dx,
\end{equation}
a coupling term, that involves integrals in $\Omega_{\ell}$ and in $\Omega_{n\ell}$ and
a different kernel $G(x,y)$,
\begin{equation}\label{coupling1.part}
 \frac{1}{2}\int_{\Omega_{n\ell}}\int_{\Omega_\ell} G(x,y)|u(y)-u(x)|^2\, dy\, dx
\end{equation}
(here $G:\Omega_\ell \times \Omega_{n\ell} \mapsto \mathbb{R}$ is assumed to be nonnegative but not necessarily symmetric; notice that the two variables belong to different sets)
and, finally, the term that involves the external source
\begin{equation} \label{source.part}
\int_\Omega f(x) u(x)\, dx.
\end{equation}

Now, let us state our hypothesis on the  involved domains and kernels. 
 With $J:\RR^N \mapsto \RR$ we denote a nonnegative measurable function such that 
 \begin{itemize}
  \item[(J1)]\emph{Visibility:} there exist $\delta>0$
and $C>0$ such that $J(z)>C$ for all $z$ such that $\|z\|\le 2\delta$.
\item[(J2)] \emph{Compactness:} the convolution type operator 
$T_J(u) (x)=\int_{\Omega_{n\ell}} J(x-y)u (y) \, dy$, defines a compact operator in $L^2(\Omega_{n\ell})$.
 \end{itemize}
In nonlocal models, $J$ is a kernel that encodes the effect of a general \emph{volumetric} nonlocal  interaction
inside the nonlocal part of the domain. 
Condition $(J1)$ guarantees the influence of nonlocality within an horizon of size at least $2\delta$ while $(J2)$ is a technical requirement fulfilled, for  instance, by continuous 
kernels, characteristic functions, or even for $L^2$ kernels, (this holds since these kernels produce Hilbert-Schmidt operators of the form 
$u \mapsto T(u)(x):=\int k(x,y) u(y) {\rm d}y$ that are compact if $k\in L^2$, see Chapter VI in \cite{brezis}). 
We also need to introduce a connectivity condition. 

 \begin{definition}\label{def:deltaconnected}
 {\rm We say that an open set $D \subset \mathbb{R}^N$ is $\delta-$connected , with $\delta\ge 0$, if it can not be written as a disjoint union of two
 (relatively) open nontrivial sets 
 that are at distance greater or equal than $\delta.$}
\end{definition}

 Notice that if a set $D$ is $\delta$ connected then it is $\delta'$ connected for any $\delta'\ge \delta$. From Definition \ref{def:deltaconnected}, 
we notice that $0-$connectedness agrees with the classical notion of being connected (in particular, open connected sets are $\delta-$connected).
Definition \ref{def:deltaconnected} can be written in an equivalent way: an open set $D$ is $\delta-$connected if given 
  two points $x,y\in D$,  there exists a finite number of points $x_0,x_1,...,x_n \in D$ 
 such that $x_0=x$,
 $x_n=y$ and $dist(x_i,x_{i+1}) <\delta$. 
 
Informally, $\delta$-connectedness combined with $(J1)$ says that the effect of nonlocality can travel beyond the horizon $2\delta$ through the whole domain.

Now we can write the following assumptions on the local/nonlocal domains: 
 \begin{enumerate}
 \item[(1)] $ \Omega_\ell$ is connected and smooth ($ \Omega_\ell$ has Lipschitz boundary),
 \item[(2)] $ \Omega_{n\ell} $ is $\delta-$connected.
 \end{enumerate}
 
Concerning the kernel involved in the coupling term, $G$, that encodes the 
interactions of $\Omega_{n\ell}$ with $\Omega_\ell$ we assume that it is given by nonnegative and measurable function 
$G: \Omega_\ell \times \Omega_{n\ell} \mapsto \mathbb{R},$
such that
 \begin{itemize}
  \item[$(G1)$]  there exist $\delta >0$
and $C>0$ such that for any $(x,y)\in \Omega_\ell \times \Omega_{n\ell}$,  $G(x,y)>C$  if $\|x-y\|\le 2\delta$.
 \end{itemize}
Finally, in order to avoid trivial couplings, we impose that  $\Omega_{\ell}$ and $\Omega_{n\ell}$ need to be closer than the horizon of the kernel involved in the coupling, we assume that
\begin{enumerate}
\item[$(P1)$]  $dist ( \Omega_{\ell},\Omega_{n\ell})<\delta $.
\end{enumerate}

\begin{remark} \label{rem.intro.dom}
Our results are valid for more general domains. In fact, we assumed that $\Omega_\ell$ is connected and that
$\Omega_{n\ell}$ is $\delta-$connected with $dist(\Omega_{\ell},\Omega_{n\ell})<\delta$, but we can also handle
the case in which $\Omega_\ell$ has several connected components and $\Omega_{n\ell}$ has several 
$\delta-$connected components as long as they are close between them. 
We prefer to state our results under conditions (1), (2), $(G1)$ and $(P1)$ just to simplify the presentation. 
\end{remark}

Now, with all these conditions at hand we go back to our energy functional \eqref{Jyvaskyla} and
look for possible critical points (minimizers). 
 Here, as usual in Neumann problems, we have to assume that
\begin{equation} \label{condicion}
\int_\Omega f (x) \, dx =0,
\end{equation}
and look for minimizers in the natural function space
$$
H_{Neu} = \left\{ u \in L^2 (\Omega): \, u|_{\Omega_\ell} \in H^1 (\Omega_\ell), \int_\Omega u (x) \, dx =0 \right\}.
$$

Let us state our first theorem.

\begin{theorem} \label{teo.1.intro} Assume that the kernels and the domains verify $(J1)$, $(J2)$, $(1)$, $(2)$, $(G1)$ and
$(P1)$. 
Given $f \in L^2 (\Omega)$ with $\int_\Omega f =0$ there exists a unique minimizer of $E_{Neu}(u)$ in $H_{Neu}$.
The minimizer of $E_{Neu}(u)$ in $H_{Neu}$ is a weak solution to the problem,
\begin{equation}
\label{eq:main.Neumann.local.intro}
\begin{cases}
\displaystyle - f(x)=\Delta u (x) + \int_{\Omega_{n\ell}} G(x,y)(u(y)-u(x))\, dy,\ &x\in \Omega_\ell,
\\[10pt]
\displaystyle \frac{\partial u}{\partial \eta} (x)=0,\qquad & x\in \partial \Omega_\ell, 
\end{cases}
\end{equation}
in the local domain $\Omega_\ell$
together with a nonlocal equation with a source in $\Omega_{n\ell}$,
\begin{equation}
\label{eq:main.Neumann.nonlocal.intro.888}
\displaystyle - f(x) = 2 \int_{\Omega_{n\ell}}\!\! J(x-y)(u(y)-u(x))\, dy,
+ \int_{\Omega_\ell} G(x,y)(u(y)-u(x))\, dy ,\qquad x \in \Omega_{n\ell}.
\end{equation}
\end{theorem}

Notice that in the resulting equations the coupling terms between the local and the nonlocal
regions appear as source integral terms in the corresponding equations. 

The key to obtain this result will be to prove a Poincar\'e-Wirtinger type inequality: there exists $c>0$
such that 
$$
\begin{array}{l}
\displaystyle
\int_{\Omega_\ell} |\nabla u|^2 + \frac{1}{2 }\int_{\Omega_{n\ell}}\int_{\Omega_{n\ell}} J(x-y)(u(y)-u(x))^2\, dy\, dx
+ \frac{1}{2}\int_{\Omega_{n\ell}}\int_{\Omega_\ell} G(x,y)|u(y)-u(x)|^2\, dy\,  dx \\[10pt]
\qquad \qquad \qquad \displaystyle \geq \displaystyle c\int_{\Omega}u^2 (x) \, dx
\end{array}
$$
for every function $u\in H_{Neu}$. Here the constant $c>0 $ can be estimated in terms of the
parameters of problem (the geometry of the domains and the kernels). 

{\subsection{Second model. Mixed couplings.} Now, let us present our second model. 
In the first model we have a nonlocal volumetric coupling between $ \Omega_{\ell}$ and $ \Omega_{n\ell}$. 
In the second model we introduce \emph{mixed} couplings. In mixed couplings volumetric and lower dimensional parts can interact with each other.
The mixed couplings that we introduce involve interactions of $\Omega_{n\ell}$ with a fixed  smooth hypersurface  $$\Gamma \subset \overline{\Omega}_\ell.$$ 
 For $
u: \Omega \mapsto \mathbb{R}$
we consider the energy
\begin{equation} \label{energy.909}
\begin{array}{l}
     F_{Neu} (u)  \displaystyle:=\frac{1}{2}\int_{\Omega_{l}} | \nabla u (x) |^2 dx + 
    \frac{1}{2}\int_{\Omega_{nl}}\int_{\Omega_{nl}}J(x-y)\left(u(y)-u(x)\right)^2 dy\, dx \\[10pt]
    \displaystyle \quad \qquad \qquad\qquad \qquad + \frac{1}{2}\int_{\Omega_{nl}}
    \int_{\Gamma}G(x,z)\left(u(x)-u(z)\right)^2 d\sigma(z) dx
    - \int_\Omega f (x) u(x) \, dx.
\end{array}
\end{equation}
Here, we can identify again the local part of the energy acting in $\Omega_\ell$, \eqref{local.part}, 
the nonlocal part acting in $\Omega_{n\ell}$, \eqref{nonlocal.part}, and the external source
\eqref{source.part};
but now the coupling term is different (now it involves integrals in $\Omega_{n\ell}$ and on the surface
$\Gamma \subset \overline{\Omega_{\ell}}$),
$$
 \frac{1}{2}\int_{\Omega_{nl}}\int_{\Gamma}G(x,z)\left(u(x)-u(z)\right)^2 \, d\sigma(z)\,  dx.
$$

In this context, we assume the same conditions on $J$ as for the first model.  
Concerning the coupling, a nonnegative and measurable function 
$G:\Gamma \times \Omega_{n\ell} \mapsto \mathbb{R}$
plays the role of the associated kernel. The following condition is analogous to the volumetric counterpart
 \begin{itemize}
  \item[$(G2)$]  there exist $\delta>0$
and $C>0$ such that for any $(x,y)\in \Gamma\times \Omega_{n\ell}$,  $G(x,y)>C$  if $\|x-y\|\le 2\delta$.
 \end{itemize}
Again, to avoid trivial couplings we assume that 
\begin{enumerate}
\item[$(P2)$]  $dist(\Gamma,\Omega_{n\ell})<\delta$.
\end{enumerate}

Here again we have to assume that
$
\int_\Omega f(x) \, dx =0,$
and then we look for minimizers in
$
H_{Neu} = \left\{ u \in  L^2 (\Omega) : \, u|_{\Omega_\ell} \in H^1 (\Omega_\ell), \, \int_\Omega u(x) \, dx =0 \right\}.$
As before, we can obtain existence and uniqueness of a minimizer of $F_{Neu}(u)$ in $H_{Neu}$.

\begin{theorem} \label{teo.2.intro} Assume that the kernels and the domains verify $(J1)$, $(J2)$, 
$(1)$, $(2)$, $(G2)$ and
$(P2)$. 
Given $f \in L^2 (\Omega)$ with $\int_{\Omega} f =0$ there exists a unique minimizer of $F_{Neu}$ 
in $H_{Neu}$. The minimizer of $F_{Neu}$ is a weak solution to 
 \begin{equation}
 \label{local-N.intro}
 \left\{
\begin{array}{ll}
 \displaystyle - f(x)  =  \Delta u (x),  \qquad & x\in \Omega_\ell, \\[10pt]
 \displaystyle   \frac{\partial u}{\partial \eta}(x)  =  0, \qquad & x\in \partial \Omega_\ell \setminus \Gamma,  \\[10pt]
 \displaystyle    \frac{\partial u}{\partial \eta}(x)  = \int_{\Omega_{nl}} G(y,x)( u(y)-  u(x)) dy, 
 \qquad & x \in \Gamma,
 \end{array} \right.
\end{equation}
and
\begin{equation}\label{nonlocal-N.intro} 
- f(x)  = 2 \int_{\Omega_{nl}} J(x-y)\left(u(y)-u(x) \right)dy 
- \int_\Gamma G(x,z) (u(x)-  u(z)) d\sigma (z), \ x \in \Omega_{nl}.
   \end{equation} 
\end{theorem}

Notice that now the coupling term between the local and the nonlocal
regions in the local part of the problem appears as a flux condition
on $\Gamma$.

\subsection{Singular kernels.}
We can also deal with singular kernels related to the fractional Laplacian and consider
energies of the form
\begin{equation} \label{energy.sing}
\begin{array}{l}
\displaystyle 
E_{Neu}(u):=\int_{\Omega_\ell} \frac{|\nabla u (x)|^2}{2} \, dx 
+ \frac{1}{2}\int_{\Omega_{n\ell}}\int_{\Omega_{n\ell}}  \frac{C}{|x-y|^{N+2s}} |u(y)-u(x)|^2\, dy \, dx
\\[10pt]
\displaystyle \qquad\qquad  \qquad
+ \frac{1}{2}\int_{\Omega_{n\ell}}\int_{\Omega_\ell} G(x,y)|u(y)-u(x)|^2\, dy\,  dx
- \int_\Omega f (x) u (x) \, dx.
\end{array}
\end{equation}

Here we look for minimizers in the space
$$
H_{Neu} = \left\{ u : u|_{\Omega_\ell} \in H^1 (\Omega_\ell), \ u|_{\Omega_{n\ell}} 
\in H^s (\Omega_{n\ell}), \ \int_\Omega u(x) \, dx =0 \right\}.
$$

This case is much simpler since we have the compact embeddings $H^1 (\Omega_\ell)
\hookrightarrow L^2 (\Omega_\ell)$ and $H^s (\Omega_{n\ell})
\hookrightarrow L^2 (\Omega_{n\ell})$. 
Also for these energies we can show the following result 
(that is the fractional counterpart to our previous results).

\begin{theorem} \label{teo.1.intro.sing}
Given $f \in L^2 (\Omega)$ with $\int_\Omega f =0$ there exists a unique minimizer of $E_{Neu}(u)$ in $H_{Neu}$.
The minimizer of $E_{Neu}(u)$ in $H_{Neu}$ is a weak solution to the equation
\begin{equation}
\label{eq:main.Neumann.local.intro.sing}
\left\{
\begin{array}{ll}
\displaystyle - f(x)=\Delta u (x) + \int_{\Omega_{n\ell}} G(x,y)(u(y)-u(x))\, dy,\ &x\in \Omega_\ell,
\\[10pt]
\displaystyle \frac{\partial u}{\partial \eta} (x)=0,\qquad & x\in \partial \Omega_\ell
\end{array} \right.
\end{equation}
and a nonlocal equation with a source in $\Omega_{n\ell}$,
\begin{equation}
\label{eq:main.Neumann.nonlocal.intro.sing}
\displaystyle - f(x) = 2 \int_{\Omega_{n\ell}}  \frac{C}{|x-y|^{N+2s}} (u(y)-u(x))\, dy  
 + \int_{\Omega_\ell} G(x,y)(u(y)-u(x))\, dy, \quad x \in \Omega_{n\ell}.
\end{equation}
\end{theorem}

We can also deal with mixed couplings and look for minimizers of
\begin{equation} \label{energy.909.sing}
\begin{array}{l}
    F_{Neu} (u)  \displaystyle:=\frac{1}{2}\int_{\Omega_{l}} | \nabla u (x) |^2 dx + 
    \frac{1}{2}\int_{\Omega_{nl}}\int_{\Omega_{nl}} \frac{C}{|x-y|^{N+2s}}
    \left(u(y)-u(x)\right)^2 dydx \\[10pt]
    \displaystyle \quad \qquad \qquad\qquad \qquad + \frac{1}{2}\int_{\Omega_{nl}}\int_{\Gamma}G(x,z)\left(u(x)-u(z)\right)^2 d\sigma(z) dx
    - \int_\Omega f(x) u(x) \, dx.
\end{array}
\end{equation}

\begin{theorem} \label{teo.2.intro.sing}
Given $f \in L^2 (\Omega)$ with $\int_{\Omega} f =0$ there exists a unique minimizer of $F_{Neu}(u)$ 
in $H_{Neu}$. The minimizer of $F_{Neu}$ is a weak solution to 
 \begin{equation}\label{local-N.intro.sing}
 \left\{
\begin{array}{ll}
 \displaystyle  - f(x)  =  \Delta u (x),  \qquad & x\in \Omega_\ell, \\[10pt]
 \displaystyle   \frac{\partial u}{\partial \eta}(x)  =  0, \qquad & x\in \partial \Omega_\ell \setminus \Gamma,  \\[10pt]
 \displaystyle    \frac{\partial u}{\partial \eta}(x)  = \int_{\Omega_{nl}} G(y,x)( u(y)-  u(x)) dy, \qquad & x \in \Gamma, 
    \end{array} \right.
    \end{equation}
    and
    \begin{equation}\label{nonlocal-N.intro.888} 
- f(x)  = 2 \int_{\Omega_{n\ell}}  \frac{C}{|x-y|^{N+2s}} (u(y)-u(x)) \, dy  
   -  \int_\Gamma G(x,z) ( u(x)-  u(z)) d\sigma (z) ,\qquad  x \in \Omega_{nl}. 
   \end{equation} 
\end{theorem}

\begin{remark} 
The kernel $G$ can also be a singular kernel
$$
G(x,z) = \frac{c}{|x-z|^{N+2t}}.
$$
In this case one has to add the condition
$$
\int_{\Omega_{n\ell}}\int_{\Omega_\ell} G(x,y)|u(y)-u(x)|^2\, dy\,  dx < +\infty
$$
in the space $H_{Neu}$. The proof that there is a minimizer is similar and hence
we leave the details to the reader. 
\end{remark}

\begin{remark} When $s>1/2$ since we have a trace theorem for $H^s (\Omega_{n\ell})$
we can also deal with surface to surface couplings and study
\begin{equation} \label{energy.909.sing.surface}
\begin{array}{l}
    F_{Neu} (u)  \displaystyle:=\frac{1}{2}\int_{\Omega_{l}} | \nabla u |^2 dx + 
    \frac{1}{2}\int_{\Omega_{nl}}\int_{\Omega_{nl}} \frac{C}{|x-y|^{N+2s}}
    \left(u(y)-u(x)\right)^2 dydx \\[10pt]
    \displaystyle \quad \qquad \qquad\qquad \qquad 
    + \frac{1}{2}\int_{\Gamma_\ell}\int_{\Gamma_{n\ell}}G(x,z)\left(u(x)-u(z)\right)^2 d\sigma(z) d \sigma (x)
    - \int_\Omega f(x) u(x) \, dx.
\end{array}
\end{equation}
Here $\Gamma_\ell$ is a smooth hypersurface in $\overline{\Omega_\ell}$ and
$\Gamma_{n\ell}$ is a smooth hypersurface in $\overline{\Omega_{n\ell}}$. Notice that
now we are coupling the domains via interactions on hypersurfaces. 
The hypothesis on the kernel $G$ that is needed to obtain a coupling between the two
domains is, as before, some strict positivity on pairs of points belonging to the coupling surfaces.
\end{remark}

Let us end the introduction with a brief description of previous references.
From a mathematical point of view, interesting problems arise from coupling local and nonlocal models, 
see \cite{Peri2,Peri3,delia2,delia3,Du,Gal,GQR,Kri,santos2020} and references therein.
As previous examples of coupling approaches between 
local and nonlocal regions we refer the reader to \cite{Peri1,Peri2,Peri3,delia2,delia3,delia,Du,Gal,GQR,Han,Kri,santos2020,Sel,Sel2,Sel3}
the survey \cite{SUR}
and references therein. 
Previous strategies treat the coupling condition as an optimization
problem (the
goal is to minimize the mismatch of the local and nonlocal solutions in a common overlapping
region).
Another possible strategy for coupling relies on the partitioned procedure as a general coupling
strategy for heterogeneous systems, the system is divided into sub-problems in
their respective sub-domains, which communicate with each other via transmission
conditions. Moreover, couplings between sets of different dimension are possible.
In \cite{Bere} the effects of network transportation on enhancing biological invasion is studied. The proposed mathematical model consists of one equation with nonlocal diffusion in a one-dimensional domain coupled via the boundary condition with a standard reaction-diffusion in a two-dimensional domain. 
In \cite{delia2}, local and nonlocal problems were coupled through a prescribed region in which both kinds of equations overlap (the value of the solution in the nonlocal part of the domain is used as a Dirichlet boundary condition for the local part and vice-versa). This kind of coupling gives continuity of the solution in the overlapping region but does not preserve the total mass when Neumann boundary conditions are imposed.
In \cite{delia2} and \cite{Du}, numerical schemes using local and nonlocal equations were developed and used to improve the computational
accuracy when approximating a purely nonlocal problem. In \cite{GQR} and \cite{santos2020} (see also \cite{Gal, Kri}), 
evolution problems related to energies closely related to ours are studied (here we deal with stationary problems).

\medskip

The paper is organized as follows: In Section \ref{sect-volum} we deal 
with our first model involving volumetric coupling and prove 
Theorem \ref{teo.1.intro} while in Section \ref{sect-surface} we prove the results concerning the second
model, Theorem \ref{teo.2.intro}. Finally, in Section \ref{sect-sing} we deal with singular kernels.

\section{Volumetric coupling}  \label{sect-volum}

{\bf First model. Coupling local/nonlocal problems via source terms.}
Our aim is to look for a minimizer of the energy 
$$
\begin{array}{l}
\displaystyle 
E_{Neu}(u):=\int_{\Omega_\ell} \frac{|\nabla u|^2}{2} 
+ \frac{1}{2}\int_{\Omega_{n\ell}}\int_{\Omega_{n\ell}} J(x-y)|u(y)-u(x)|^2\, dy \, dx
\\[10pt]
\displaystyle \qquad\qquad \qquad  + \frac{1}{2}\int_{\Omega_{n\ell}}\int_{\Omega_\ell} G(x,y)|u(y)-u(x)|^2\, dy\,  dx
- \int_\Omega f(x)u(x) \, dx
\end{array}
$$
in the space $$
H_{Neu} = \left\{ u \in L^2 (\Omega), u|_{\Omega_\ell} \in H^1 (\Omega_\ell), \int_\Omega u =0 \right\},
$$
assuming
$$
\int_\Omega f (x) \, dx =0.
$$

Let us first prove an auxiliary lemma.  
 
\begin{lemma}
\label{lema:J0implica_cte}
 Let $D$ be an open  $\delta$ connected set, and $u:D\to \RR$. If
 $$
 \int_{D}\int_{D}J(x-y)|u(x)-u(y)|^2 \, dx \, dy=0,
 $$
 then $u(x)=k$ a.e. $x\in D$. 
\end{lemma}
\begin{proof}
 Pick $x_0\in D$ and a ball $B_0=B_{\delta}(x_0)$,  we have 
 $$
 C\int_{D\cap B_0 }\int_{D\cap B_0}(u(x)-u(y))^2\le 
 \int_{D\cap B_0 }\int_{D\cap B_0}J(x-y)(u(x)-u(y))^2=0,
 $$
(since $J(x-y)>C$ for $x,y\in B_0$) and  hence  $u(x)=k_0$ a.e. $x\in D\cap B_0$. In order to see that this property holds a.e. $x\in D$, let us introduce the set  $\mathcal{M}=\{A\subset D, A\,\, \mbox{open}: u(x)=k_0 \,\mbox{a.e.} \,x\in A \}$ with the partial order given by $\subset$. Since $\mathcal{M}\neq\emptyset$ there exists a maximal open set $M\in \mathcal{M}$. If $M\subsetneq \Omega$ then we consider the set $\emptyset\neq D\setminus M$. 
 
 If $D\setminus M$ is open we necessarily have that $dist(M,D\setminus M)<\delta$
 (here we are using that $D$ is $\delta$ connected). If
 $D\setminus M$ is not open, then $dist(M, D\setminus M)=0$  (since $D$ is open).  
Either case, there exists a ball  $B_1$ of radius $\delta$ such that $B_1\cap D\setminus M\neq \emptyset$ and  $B_1\cap M$ has positive measure (since both, $B_1$ and $M$, are open sets).   Arguing as before we see that $u(x)=k_0$ a.e. $x\in B_1\cap D$, a contradiction (since $M$ is maximal with that property and we would have $M\subsetneq M\cup (B_1\cap D)=M$).
 We see that $M=D$ and the proof is complete. 
\end{proof}

\begin{lemma}
\label{lema:contagio_escalar}
Let $u_n:\Omega\to \RR$ be a sequence such that $u_n\to 0$ strongly in $L^2(\Omega_{\ell})$ and weakly in $L^2(\Omega_{n\ell})$, if in addition
\begin{equation}
\label{eq:nolcalomega}
\lim_{n\to\infty}\int_{\Omega_{n\ell}}\int_{\Omega} J(x-y)(u_n(x)-u_n(y))^2dy dx=0,
\end{equation}
 then
 $$
 \lim_{n\to\infty}\int_{\Omega_{n\ell}}u_n(x)^2 dx=0.
 $$
\end{lemma}
\begin{proof}
From \eqref{eq:nolcalomega},   the convergence of $\{u_n\}$ and property (J2),  we easily find that 
\begin{equation}
 \label{eq:solonl}
 \lim_{n\to\infty}\int_{\Omega_{n\ell}}\int_{\Omega_{\ell}} J(x-y)u_n(x)^2dy dx=0.
\end{equation}
Let us define  
 $$
 A^0_\delta=\Big\{x\in \Omega_{n\ell}: dist(x,\Omega_{\ell})<\delta \Big\}.
 $$
Notice that thanks to property $(P1)$ and to the fact that  $\Omega_{n\ell}$ is open we see that  $A^0_\delta$ is open and non empty.  In particular it has positive n-dimensional measure.   For any $x\in \overline{A^0_{\delta}}$
we consider the continuous and \emph{strictly} positive function  $g(x)=|B_{2\delta}(x)\cap \Omega_{\ell}|$. Since 
$\overline{A^0_{\delta}}$ is a compact set, there exists a constant $m>0$ such that $g(x)\ge m$ for any  $x\in \overline{A^0_{\delta}}$. 
 As a consequence 
  $$
 \int_{\Omega_{n\ell}}\int_{\Omega_{\ell}} J(x-y)u_n(x)^2dy dx \ge 
 \int_{A^0_\delta}\int_{B_{2\delta}(x)\cap \Omega_{\ell}} J(x-y)u_n(x)^2dy dx \ge
 mC\int_{A^0_\delta}u_n(x)^2 dx.
 $$ 
 and therefore, thanks to  \eqref{eq:solonl}, $u_n\to 0$ in $L^2(A^0_\delta)$.
 In order to iterate 
 this argument we notice that at this point we know that 
 $u_n\to 0$ strongly in $A^0_\delta$ and weakly in $\Omega_{n\ell} \setminus \overline{A^0_\delta}$, hence again from \eqref{eq:nolcalomega} we get
\begin{equation}
 \label{eq:soloadelta}
 \lim_{n\to\infty}\int_{\Omega_{n\ell} \setminus \overline{A^0_\delta}}\int_{ A^0_\delta} J(x-y)u_n(x)^2dy dx=0.
\end{equation}
Since $\Omega_{n\ell}$ is $\delta$ connected, $dist(\Omega_{n\ell} \setminus \overline{A^0_\delta}, A^0_\delta)<\delta$. Considering now
 $$
 A^1_\delta=\{x\in \Omega_{n\ell}\setminus \overline{A^0_\delta}: dist(x,A^0_\delta)<\delta\},
 $$
 and proceeding as before, we obtain, from \eqref{eq:soloadelta}, that 
 $u_n\to 0$ strongly in $A^1_{\delta}$. This argument can be repeated and giving strong converge in $L^2(A^j_\delta)$ for 
 $$
 A^j_\delta=\left\{x\in \Omega_{n\ell}\setminus\overline{ \bigcup_{0\le i< j} A^i_\delta}: dist \Big(x, \bigcup_{0\le i< j} A^i_\delta \Big)<\delta\right\}.
 $$
 Since $\Omega_{n\ell}$ is bounded, we have, for a finite number $J\in \N$, $$\Omega_{n\ell}=\bigcup_{0\le i< J} A^i_\delta$$ and therefore  the proof is complete.  \end{proof}

As we have mentioned in the Introduction our goal is to show that given $f \in L^2 (\Omega)$ with $\int_\Omega f =0$ there exists a unique minimizer of the energy functional $E_{Neu}(u)$ in the space $H_{Neu}$.

To use the direct method of calculus of variations to obtain the result we have to show that
$E_{Neu}(u)$ is coercive and weakly lower semicontinuous.

To this end we prove a key result, a Poincar\'e-Wirtinger type inequality holds.

\begin{lemma}\label{lema.Poincare} There exists $c>0$ such that 
$$
\begin{array}{l}
\displaystyle
\int_{\Omega_\ell} |\nabla u|^2 + \frac{1}{2 }\int_{\Omega_{n\ell}}\int_{\Omega_{n\ell}} J(x-y)(u(y)-u(x))^2\, dy\, dx
+ \frac{1}{2}\int_{\Omega_{n\ell}}\int_{\Omega_\ell} G(x,y)|u(y)-u(x)|^2\, dy\,  dx \\[10pt]
\qquad \qquad \qquad \displaystyle \geq \displaystyle c\int_{\Omega}u^2 (x) \, dx
\end{array}
$$
for every function $u\in H_{Neu}$.
\end{lemma}

\begin{proof} 
First, let us point out that when $G(x,y) = J(x-y)$, $\Omega_\ell\subset \Omega$ is a smooth 
convex subdomain we can use a result from \cite{GQR}. 
In this case the norm in $H_{Neu}$
bounds a pure nonlocal seminorm.  
There exists $c >0$ such that,
\begin{equation}
\label{eq:fundamental.inequality}
\displaystyle
\int_{\Omega_\ell} |\nabla u|^2 + \frac{1}{2 }\int_{\Omega_{n\ell}}\int_{\Omega} J(x-y)(u(y)-u(x))^2\, dy \, dx
\geq \displaystyle c\int_{\Omega}\int_{\Omega} J(x-y)(u(y)-u(x))^2\, dy\, dx,
\end{equation}
for every $u\in\mathcal{H}:=\{u\in L^2(\Omega): u|_{\Omega_{\ell}}\in H^1(\Omega_{\ell}) \}$. 
Then the desired Poincar\'e-Wirtinger inequality follows from an analogous inequality for the purely nonlocal energy:  there exists $c>0$ such that
$$
\displaystyle
\int_{\Omega}\int_{\Omega} J(x-y)(u(y)-u(x))^2\, dy \, dx
\geq \displaystyle c \int_{\Omega}u^2 (x) \, dx
$$
for every function $u\in\mathcal{H}$ such that $\int_\Omega u =0$; see~\cite{ElLibro}.

To obtain the Poincar\'e-Wirtinger inequality in the general case we argue by contradiction. 
Assume that there is a sequence $u_n$ such that
\begin{equation} \label{eq.gradientes}
\int_{\Omega_\ell} |\nabla u_n|^2 (x) \, dx \to 0,
\end{equation}
\begin{equation} \label{eq.J}
\frac{1}{2 }\int_{\Omega_{n\ell}}\int_{\Omega_{n\ell}} J(x-y)|u_n(y)-u_n(x)|^2\, dy\, dx \to 0,
\end{equation}
\begin{equation} \label{eq.G}
\frac{1}{2}\int_{\Omega_{n\ell}}\int_{\Omega_\ell} G(x,y)|u_n(y)-u_n(x)|^2\, dy\,  dx \to 0,
\end{equation}
\begin{equation} \label{eq.=1}
\int_{\Omega}|u_n|^2 (x) \, dx = 1
\end{equation}
and 
\begin{equation} \label{eq.int-0}
\int_\Omega u_n (x) \, dx =0.
\end{equation}

Since the $L^2$ norm is bounded we can extract a subsequence such that
$$
u_n \rightharpoonup u \qquad \mbox{weakly in } H^1 (\Omega_{\ell}),  
$$
then, from \eqref{eq.gradientes} we get that there is a constant $k_1$ such that
$$
u_n \to k_1 \qquad \mbox{strongly in } H^1 (\Omega_{\ell}).  
$$
Now, from \eqref{eq.J} (using again that the $L^2$ norm is bounded) we can refine the subsequence
to obtain that 
$$
u_n \rightharpoonup k_2 \qquad \mbox{weakly in } L^2 (\Omega_{n\ell}),  
$$
for some constant $k_2$. 

From \eqref{eq.G} we conclude that 
$$
\frac{1}{2}\int_{\Omega_{n\ell}}\int_{\Omega_\ell} G(x,y)|k_1-k_2|^2\, dy\,  dx
\leq \liminf_{n \to \infty}  \frac{1}{2}\int_{\Omega_{n\ell}}\int_{\Omega_\ell} G(x,y)|u_n(y)-u_n(x)|^2\, dy\,  dx =0,
$$
and hence
$$
k_1 = k_2. 
$$
We have
$$
|\Omega_\ell | k_1 = \lim_n \int_{\Omega_\ell} u_n (x) \, dx  
$$
and 
$$
|\Omega_{n\ell} | k_2 = \lim_n \int_{\Omega_{n\ell}} u_n (x) \, dx .
$$
Then, adding the previous identities and using that $k_1=k_2$, from \eqref{eq.int-0}, we obtain
$$
k_1=k_2=0. 
$$
From the strong convergence in $H^1(\Omega_\ell)$ we obtain 
$$
 \lim_n \int_{\Omega_\ell} |u_n|^2 (x) \, dx  =0
$$
and hence from \eqref{eq.=1} we get
$$
\int_{\Omega_{n\ell}}|u_n|^2 (x) \, dx \to 1.
$$
This contradicts the result in \cite{ElLibro} since there exists $c>0$ such that
$$
\displaystyle
0= \lim_n \int_{\Omega_{n\ell}}\int_{\Omega_{n\ell}} J(x-y)(u_n(y)-u_n(x))^2\, dy \, dx
\geq \displaystyle c  \lim_n \int_{\Omega_{n\ell}}\left| u_n (x) -  \dashint_{\Omega_{n\ell}} u_n    \right|^2 \, dx = 1.
$$
This contradiction finishes the proof.
\end{proof}

The proof of Lemma \ref{lema.Poincare} is made by contradiction and therefore
it does not provide an estimate for the constant. In what follows we present an
alternative proof that provides an explicit constant in terms of the relevant
quantities, the geometries of the involved domains $\Omega_\ell$ and 
$\Omega_{n\ell}$, and the kernels $J$ and $G$ (in fact the constant depends on 
the parameter $\delta $ and $\min J$, $\min G$ for points $(x,y)$ such that
$|x-y| < 2 \delta$).  In order to do so, we introduce the following remark. 

 \begin{remark}[Geometric structure of bounded $\delta-connected$ sets]
 \label{rem:trees}
 Let $D$ a bounded $\delta-connected$ set. Boundedness implies that it is possible to find a \emph{finite} collection $\mathcal{C}$ of open sets, each of them of diameter less than or equal to $\delta/2$, such that  
 $$
 D= \bigcup_{B\in \mathcal{C}} B.
 $$  
 We introduce a tree structure in $\mathcal{C}$ in the following way: for a fixed $B\in \mathcal{C}$,
 that will be called the root of the tree, we set $B_{0}=B$. Now we pick $B_1\neq B_0$, $B_1\in \mathcal{C}$ such that 
 $dist(B_{0}, B_{1})<\delta$ (if more than one set has this property we choose any of them). Now we proceed in the same way, picking $B_2\in \mathcal{C}\setminus\{B_0,B_1\}$ such that $dist(B_{1},B_{2})<\delta$ and repeat the procedure until we reach a set $B_k$ such that either $\emptyset=\mathcal{C}\setminus\{B_0,B_1,\cdots, B_k\}$  or $\emptyset\neq \mathcal{C}\setminus\{B_0,B_1,\cdots, B_k\}$ and all the elements in $\mathcal{C}\setminus\{B_0,B_1,\cdots, B_k\}$ 
 are at distance greater or equal than $\delta$ from $B_k$. In the first case we are done and  we introduce a (total) order in $\mathcal{C}$ given by $B_i<B_j$ if $B_j$ was chosen later than $B_i$. In the second case, we call the obtained totally ordered set $\mathcal{C}^0_1$ a branch with root $B_0$. Then we continue with further branches. First, repeating our procedure, we check if with the remaining elements $\mathcal{C}\setminus \mathcal{C}^0_1$, it is possible to get another branch $\mathcal{C}^0_2$ with the same root $B_0$. When all branches with root $B_0$ are exhausted, we look within the set of remaining elements for potential branches with root $B_1$, $B_2$, etc. until we reach $B_{k-1}$. At this point, and  iteratively, all elements of all generated branches are tested as roots of new branches. Since $\mathcal{C}$ is finite, this procedure reaches an end in a finite number of steps yielding a partially orderer set $\mathcal{T}$.
 Naturally $\mathcal{T}\subset \mathcal{C}$, we claim that $\mathcal{T}= \mathcal{C}$. Otherwise, we consider the non empty open sets 
 $D_1=\cup_{B\in \mathcal{C}\setminus \mathcal{T}}B$ , $D_2=\cup_{B\in \mathcal{T}}B$ and then $D=D_1\cup D_2$ with $dist(D_1,D_2)>\delta$ a contradiction since $D$ is $\delta-connected$.
 
In $\mathcal{T}$ we define the natural \emph{partial} order and  $\mathcal{T}$ is called a $\delta-tree$. Notice that for two 
\emph{consecutive} elements $B<\tilde B$ we have $diam(B,\tilde B)<2\delta$. In particular, from our hypothesis on the kernel $J$, this implies that
$J(x-y)>0$ for $x\in B$, $y\in \tilde B$.

 For further use we introduce the following notations: we let the degree of a root as the number of branches emerging from that root while the degree of $\mathcal{T}$ is the maximum degree of all the roots belonging to $\mathcal{T}$. Finally, the length of a branch is the cardinal of that branch excluding its root.
 \end{remark}

 \begin{remark}
 \label{rem:cubo}
 In simple cases it is possible to easily characterize potential tree structures on domains.
 In a square $D \subset \RR^2$ (or a cube in $\RR^N$) we can build $\mathcal{T}$ with a single branch. The cardinal of $\mathcal{T}$ is bounded by $\frac{8^2|D|}{\delta^2}$ (or by $\frac{8^N|D|}{\delta^N}$ in $\RR^N$). Moreover, even rather complicated structures can be covered by simple trees if the scale related to 
 $\delta$ is large enough as it is shown in Figures \ref{fig.1} and \ref{fig.2}. 
  \end{remark}
 \begin{figure}
\includegraphics[scale=.25]{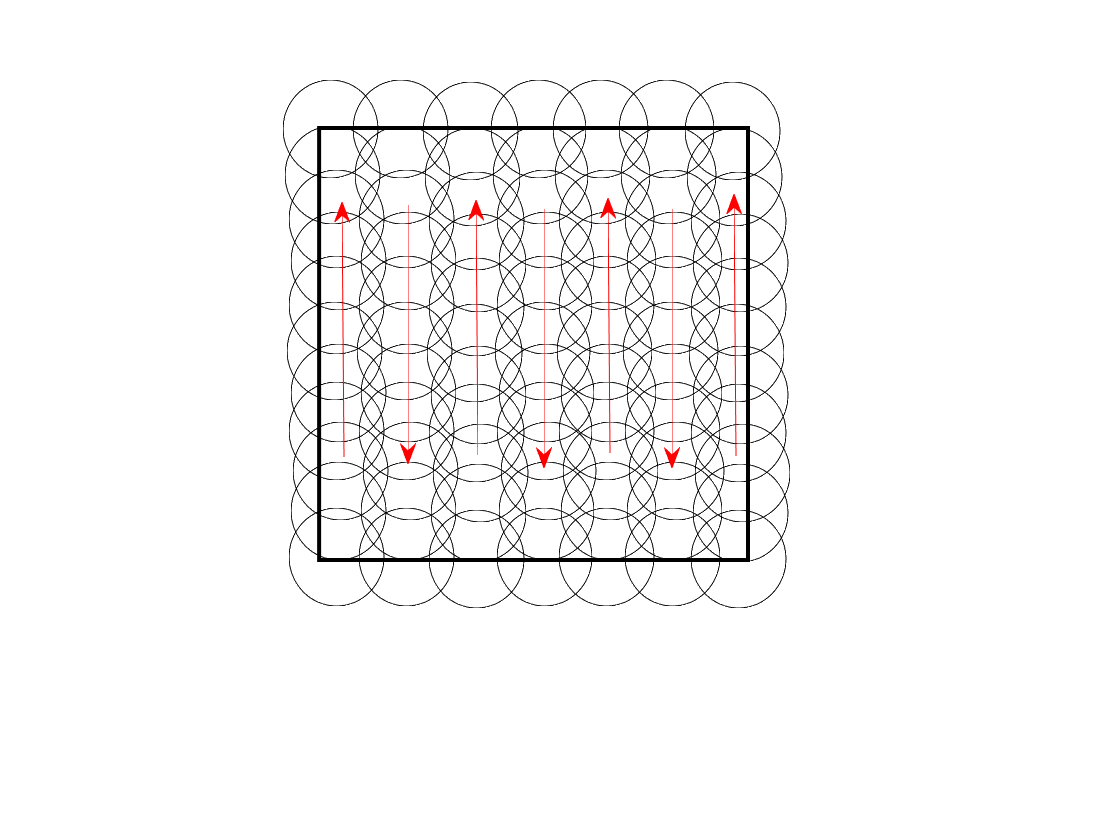}
\caption{A tree with a single branch in a square $\Omega$. In this example a branch is built with a collection of sets $\mathcal{C}=\{B_i\}$ given by the intersection of  $\Omega$ with balls of diameter $\delta/2$. A possible ordering is given by the path along the red arrows. }
 \label{fig.1}
\end{figure}
 \begin{figure}
\includegraphics[scale=.25]{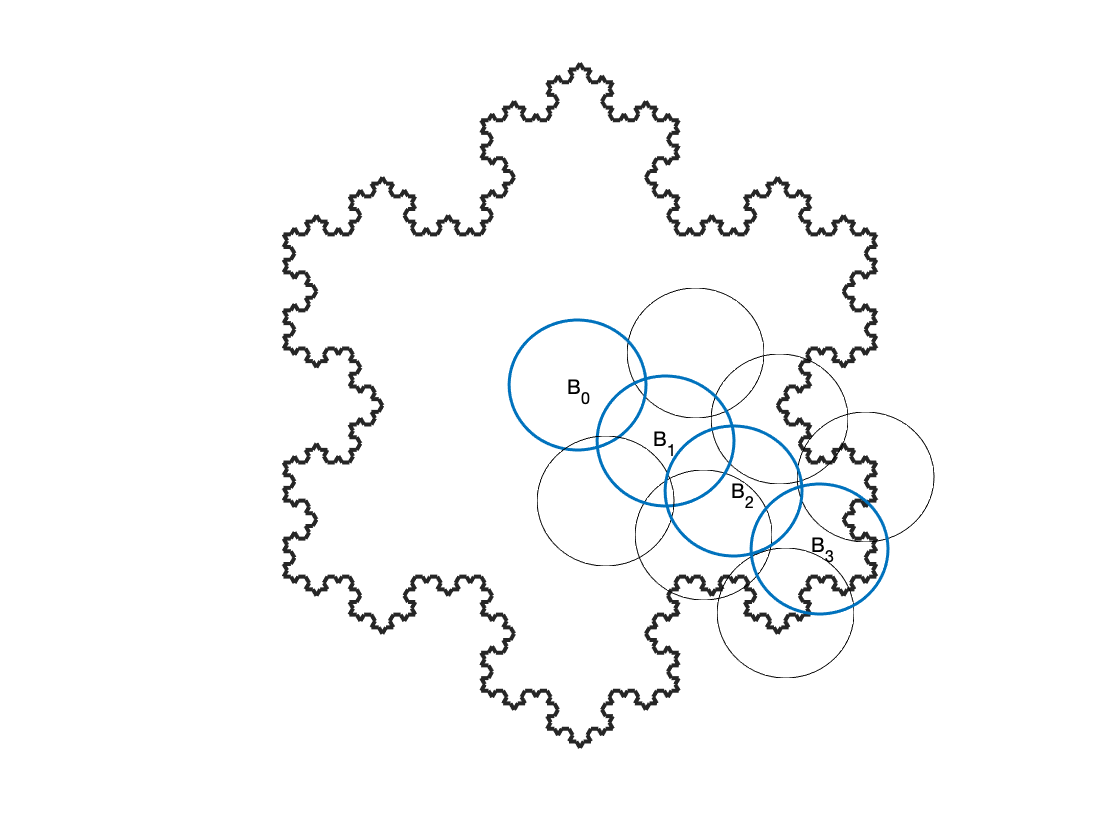}
\caption{A fractal like domain with a simple tree structure with root $B_0$. Only a single branch $\mathcal{C}_0^1$ is shown in the figure together with a few neighboring elements of $\mathcal{C}_1^0$. Clearly, details smaller than the "scale" size $\delta$ are ignored by the tree structure. }
 \label{fig.2}
\end{figure}

With this structure at hand, let us prove the following lemma.

\begin{lemma}
\label{lemma:branch}
Let $B^*=\{B_0,B_1,\cdots,B_N\}$, a finite collection of  open sets $B_i\subset \RR^n$, such that
$diam(B_i\cup B_{i+1})<2\delta$.  Then, for any $u:\cup_{j} B_j\to \RR$, it holds
$$
\begin{array}{l}
\displaystyle 
\sum_{k=0}^N\int_{B_k} |u(x)|^2 \, dx  
\le \left( \sum_{k=0}^N 2^k\frac{|B_k|}{|B_0|} \right)\int_{B_0}u^2 \, dx \nonumber\\[10pt]
\qquad \qquad \qquad \qquad \quad \displaystyle  + \! \frac{1}{m_J}\sum_{i=1}^{N}2^i\left(\sum_{k=i}^N\frac{|B_{k}|}{|B_{k-i}||B_{k-i+1}|}  \int_{B_{k-i}} \int_{B_{k-i+1}} \!\!\!\! \!\! J(x-y)(u(x)-u(y))^2 \, dx\, dy\right),\nonumber
\end{array}
$$
where $$0<m_J=\inf_{|z|<2\delta} J(z).$$
\end{lemma}

\begin{proof}
Since
$$
\int_{B_1} |u(x)|^2 \, dx = \frac{1}{|B_0|} \int_{B_0} \int_{B_1} |u(x)|^2 \, dx\, dy  ,
$$ 
we have
$$
\int_{B_1} |u(x)|^2 \, dx   \leq
\frac{2}{|B_0|}  \int_{B_0} \int_{B_1} |u(x)-u(y)|^2 \, dx\, dy + 2\frac{|B_1|}{|B_0|} \int_{B_0}  |u(y)|^2 \, dx\, dy ,
$$
and similarly
$$
\int_{B_k} |u(x)|^2 \, dx   \leq
\frac{2}{|B_{k-1}|}  \int_{B_{k-1}} \int_{B_k} |u(x)-u(y)|^2 \, dx\, dy + 2\frac{|B_k|}{|B_{k-1}|} \int_{B_{k-1}}  |u(y)|^2 \, dx\, dy .
$$
As a consequence, calling $\Delta=u(x)-u(y)$, we get
$$
\begin{array}{l}
\displaystyle 
\int_{B_k} |u(x)|^2 \, dx   \leq \sum_{i=1}^k\frac{2^i|B_{k}|}{|B_{k-i}||B_{k-i+1}|}  \int_{B_{k-i}} \int_{B_{k-i+1}} \Delta^2 \, dx\, dy + 2^k\frac{|B_k|}{|B_0|}\int_{B_0}u^2 \, dx,
\end{array}
$$
and therefore, we obtain,
$$
\begin{array}{ll}
\displaystyle
\sum_{k=0}^N & \displaystyle \nonumber \int_{B_k} |u(x)|^2 \, dx   
\\[10pt]
& \displaystyle \leq \left( \sum_{k=0}^N 2^k\frac{|B_k|}{|B_0|} \right)\int_{B_0}u^2 \, dx + \sum_{k=1}^{N}\left(\sum_{i=1}^k\frac{2^{i}|B_{k}|}{|B_{k-i}||B_{k-i+1}|}  \int_{B_{k-i}} \int_{B_{k-i+1}} \Delta^2 \, dx\, dy\right)  \nonumber \\[10pt]
& \displaystyle =\left( \sum_{k=0}^N 2^k\frac{|B_k|}{|B_0|} \right)\int_{B_0}u^2 \, dx + \sum_{i=1}^{N}2^i\left(\sum_{k=i}^N\frac{|B_{k}|}{|B_{k-i}||B_{k-i+1}|}  \int_{B_{k-i}} \int_{B_{k-i+1}} \Delta^2 \, dx\, dy\right).\nonumber 
\end{array}
$$
Hence, the Lemma follows by using that $1\le J(x-y)/m_J$ for any $x,y$ belonging to consecutive sets $B_i$.
\end{proof}

The previous lemma in particular says that in a branch of a tree structure of sets (as the ones described in Remark \ref{rem:trees}), the $L^2$ norm of a function defined on that branch can be bounded in terms of the $L^2$ norm on the root of the branch plus an energy involving the nonlocal operator $J$ along the branch. Taking into account that in a tree
$\mathcal{T}$ roots of branches are members of previous branches, it is clear that the $L^2$ norm of a function defined on $\mathcal{T}$ can be bounded in a similar fashion in terms of the $L^2$ norm on the root of the tree and the nonlocal energy. Moreover, if needed, explicit constants can be tracked back along the tree structure. For the sake of clarity, we do not state the next result with such a generality, although we provide below some simple examples with explicit constants.

\begin{corollary}
\label{coro:solonolocal}
For any $u\in L^2(\Omega_{n\ell})$ we consider a tree structure $\mathcal{T}$ in $\Omega_{n\ell}$ with root $B$, then there exists a constant $C=C(\mathcal{T},m_J)$ such that
$$
\int_{\Omega_{n\ell}} u^2\, dx\le C\left( \int_{B} u^2\, dx + \int_{\Omega_{n\ell}}\int_{\Omega_{n\ell}} J(x-y) (u(x)-u(y))^2 \, dx \, dy \right).
$$ 
\end{corollary}
\begin{proof} Immediate from Lemma \ref{lemma:branch}.
\end{proof}

\begin{remark}[Explicit Constants]
\label{rem:cuad1}
Notice that in some cases an explicit expression for $C(\mathcal{T},m_J)$ can be obtained.
If $\Omega_{n\ell}$ is a cube in $\RR^N$ (see Remark \ref{rem:cubo}) then we can take $\mathcal{T}$ with a single branch of length bounded by $\frac{8^N|\Omega_{n\ell}|}{\delta^N}$. In this case, using Lemma \ref{lemma:branch} and taking into account that $|B_i|=|B|=c_N\frac{\delta^N}{2^N}$ is constant (here $c_N$ the measure of the unit ball
in $\mathbb{R}^N$), together with the fact that $$\sum_{k=i}^N 2^k=2^i\sum_{k=0}^{N-i}2^k=2^i(2^{N-i+1}-1)\le 2^{N+1},$$ we can obtain the bound
$$\int_{\Omega_{n\ell}} u^2\, dx\le 2^{\frac{8^N|\Omega_{n\ell}|}{\delta^N}+1}\left( \int_{B} u^2\, dx +\frac{2^N}{c_N\delta^N m_J} \int_{\Omega_{n\ell}}\int_{\Omega_{n\ell}} J(x-y) (u(x)-u(y))^2 \, dx \, dy \right).$$
Notice that the obtained constant deteriorates as $\delta \to 0$ or $m_J \to 0$.
\end{remark}

Now, we are ready to present a proof of the  Poincar\'e-Wirtinger inequality where the constant 
can be tracked. 

\begin{lemma}\label{lema.Poincare.con.cte} 
The optimal 
constant $C^*>0$ such that 
$$
\begin{array}{l}
\displaystyle
\int_{\Omega_\ell} |\nabla u|^2 + \frac{1}{2 }\int_{\Omega_{n\ell}}\int_{\Omega_{n\ell}} J(x-y)(u(y)-u(x))^2\, dy\, dx
+ \frac{1}{2}\int_{\Omega_{n\ell}}\int_{\Omega_\ell} G(x,y)|u(y)-u(x)|^2\, dy\,  dx \\[10pt]
\qquad \qquad \qquad \displaystyle \geq \displaystyle C^*\int_{\Omega}u^2 (x) \, dx
\end{array}
$$
for every function $u\in H_{Neu}$ (that is, for $u \in \mathcal{H}$ with $\int_\Omega u =0$) can be estimated as
$$
C^* \geq c^* (J, G, \Omega_{\ell},\mathcal{T},\delta,A,B) 
$$
with $A,B$ two open sets
$A\subset \Omega_{\ell}$, $B\subset  \Omega_{n\ell}$ such that $diam(A\cup B)<2\delta$ and $\mathcal{T}$ a $\delta-tree$ structure of $ \Omega_{n\ell}$ with root $B$.
\end{lemma}

\begin{proof}
We aim to obtain a bound of the form 
$$
\begin{array}{l}
\displaystyle
\inf_{r\in \mathbb{R}} 
\int_{\Omega}|u-r|^2 (x) \, dx
\leq \frac{1}{c^*} \left[ \int_{\Omega_\ell} |\nabla u|^2 + \frac{1}{2 }\int_{\Omega_{n\ell}}\int_{\Omega_{n\ell}} J(x-y)(u(y)-u(x))^2\, dy\, dx \right. \\[10pt]
\qquad \qquad\qquad \qquad\qquad \qquad \qquad \qquad \displaystyle \left.
+ \frac{1}{2}\int_{\Omega_{n\ell}}\int_{\Omega_\ell} G(x,y)|u(y)-u(x)|^2\, dy\,  dx \right]
\end{array}
$$
for functions $u \in L^2 (\Omega)$ such that $u|_{\Omega_\ell} \in H^1 (\Omega_\ell)$. 

From our hypothesis on $G$ and the assumption $\mbox{dist}
(\Omega_\ell, \Omega_{n\ell}) <\delta$, there exist two sets 
$A\subset  \Omega_{\ell}$ and $B \subset  \Omega_{n\ell}$ such that 
$$
|x-y| \leq 2\delta, \qquad \forall x\in A, \, y\in B
$$
and therefore
$$
G(x,y) \geq m_G = \min_{|x-y|<2\delta } G(x,y), \qquad \forall x\in A, y \in B.
$$
Moreover, reducing the size of $B$ if necessary, we may assume that $diam(B)<\delta/2$ and we define $\mathcal{T}$ a $\delta-tree$  on $\Omega_{n\ell}$ with root $B$.

Let us take $r \in \mathbb{R}$ such that 
$$
\int_A (u(x) -r) \, dx =0. 
$$
Then, we have, for $v= u-r$ ­
\begin{equation}
\label{eq:intersec}
\begin{array}{l}
\displaystyle 
\frac{1}{2}\int_{\Omega_{n\ell}}\int_{\Omega_\ell} G(x,y)|v(y)-v(x)|^2\, dy\,  dx \\[10pt]
\displaystyle \geq \frac{1}{2}\int_{B}\int_{A} G(x,y)|v(y)-v(x)|^2\, dy\,  dx \\[10pt]
\displaystyle \geq m_G \frac{|A|}{2}\int_{B}|v(x)|^2\,  dx 
\end{array}
\end{equation}
where we are using that 
$$
\int_{B}|v(x)|^2\,  dx=\frac{1}{|A|^2}\int_{B}\left|\int_A(v(x)-v(y))\,dy\right|^2\,  dx \le
\frac{1}{|A|}\int_{B}\int_A|v(x)-v(y)|^2\,dy\,  dx.
$$

Now, since
$$
0<\sigma (\Omega_\ell, A) = \inf_{v \in H^1 (\Omega_\ell), \int_Av\, dx=0} 
\frac{\displaystyle \int_{\Omega_\ell} |\nabla v|^2 (x) \, dx }{\displaystyle  \int_{\Omega_\ell} v^2 (x) \, dx } ,
$$
for a suitable Poincar\'e constant $\sigma (\Omega_\ell, A) $, we have for the local region,
$$ \int_{\Omega_\ell} v^2 (x) \, dx \leq \frac{1}{\sigma (\Omega_\ell, A)} 
\int_{\Omega_\ell} |\nabla v|^2 (x) \, dx. 
$$

On the other hand, Corollary  \ref{coro:solonolocal} says that
$$
\int_{\Omega_{n\ell}} v^2\, dx\le C(\mathcal{T},m_J)\left( \int_{B} v^2\, dx + \int_{\Omega_{n\ell}}\int_{\Omega_{n\ell}} J(x-y) |v(x)-v(y)|^2 \, dx \, dy \right),
$$ 
and the lemma follows by collecting the last two inequalities together with \eqref{eq:intersec}. 
\end{proof}

\begin{remark}
\label{rem:poincare.casos}
The constant $\sigma (\Omega_\ell, A)$ can be difficult to characterize  even in the case  $A=\Omega_\ell$  (for which we just write $\sigma(\Omega_\ell)$). For some elementary domains, such as cubes or balls in $\RR^N$, explicit computations of $\sigma(\Omega_\ell)$ can be done. Moreover, in particular circumstances and for simple geometries some bounds are easy to find. Notably, for convex domains 
it is known that $$\frac{\pi^2}{\mbox{diam}^2(\Omega)}\le \sigma(\Omega_\ell),$$ 
see \cite{PW}. On the other hand, a simple argument relates $\sigma (\Omega_\ell, A)$ with $\sigma(\Omega_\ell)$ as follows: assume that $u\in H^1(\Omega_\ell)$,  $\frac{1}{|A|}\int_{A}u(x)\, dx=0$ and write 
$$
\bar u_{\Omega_\ell}=\frac{1}{|\Omega_\ell|}\int_{\Omega_\ell}u(x)\, dx=\int_{\Omega_\ell}u(x)\left(
\frac{1}{|\Omega_\ell|}-\frac{\chi_A(x)}{|A|}\right)\, dx=\int_{\Omega_\ell}(u(x)-\bar u_{\Omega_\ell})\left(
\frac{1}{|\Omega_\ell|}-\frac{\chi_A(x)}{|A|}\right)\, dx 
$$
then, Schwartz's inequality yields
$$
\bar u_{\Omega_\ell}^2\le \frac{|\Omega_\ell|}{|A|^2}\int_{\Omega_\ell}(u(x)-\bar u_{\Omega_\ell})^2\, dx \le \frac{|\Omega_\ell|}{|A|^2\sigma(\Omega_\ell)}\int_{\Omega_\ell}|\nabla u|^2(x)\, dx.
$$
Hence, we have
$$
\int_{\Omega_\ell}u(x)^2\, dx \le 2\int_{\Omega_\ell}(u(x)-\bar u_{\Omega_\ell})^2\, dx  
+2 \int_{\Omega_\ell}\bar u_{\Omega_\ell}^2\, dx \le 2\frac{1+ \left(\frac{|\Omega_\ell|}{|A|}\right)^2}{\sigma(\Omega_\ell)}\int_{\Omega_\ell}|\nabla u|^2(x)\, dx,
$$
and, as a consequence, we conclude that
$$
\frac{\sigma (\Omega_\ell)}{2\left(1+ \left( \displaystyle \frac{|\Omega_\ell|}{|A|}\right)^2 \right)}\le \sigma (\Omega_\ell, A).
$$
\end{remark}

\begin{remark}
\label{rem:cuad2}
For two adjacent square domains $\Omega_{n\ell} =[-1,0] \times [0,1]$, $\Omega_{\ell}
=[0,1] \times [0,1]$ of unitary side, Remarks \ref{rem:cuad1} and \ref{rem:poincare.casos} give
an explicit estimate of the constant. 

First, we have
$$ \int_{\Omega_\ell} v^2 (x) \, dx \leq \frac{1}{\sigma (\Omega_\ell, A)} 
\int_{\Omega_\ell} |\nabla v|^2 (x) \, dx \leq  \frac{ [\mbox{diam} (\Omega) ]^2}{2\pi^2 
\left(1+ \left( \displaystyle \frac{|\Omega_\ell|}{|A|}\right)^2 \right)} 
\int_{\Omega_\ell} |\nabla v|^2 (x) \, dx. 
$$
On the other hand, from Remark \ref{rem:cuad1} we get
$$\int_{\Omega_{n\ell}} u^2\, dx\le 2^{\frac{8^N|\Omega_{n\ell}|}{\delta^N}+1}\left( \int_{B} u^2\, dx +\frac{2^N}{c_N\delta^N m_J} \int_{\Omega_{n\ell}}\int_{\Omega_{n\ell}} J(x-y) (u(x)-u(y))^2 \, dx \, dy \right).$$

Therefore, for the particular case of two adjacent unit squares we have that
the best constant in the Poincare type inequality
$$
\begin{array}{l}
\displaystyle
\int_{\Omega_\ell} |\nabla u|^2 + \frac{1}{2 }\int_{\Omega_{n\ell}}\int_{\Omega_{n\ell}} J(x-y)(u(y)-u(x))^2\, dy\, dx
+ \frac{1}{2}\int_{\Omega_{n\ell}}\int_{\Omega_\ell} G(x,y)|u(y)-u(x)|^2\, dy\,  dx \\[10pt]
\qquad \qquad \qquad \displaystyle \geq \displaystyle C^* \int_{\Omega}u^2 (x) \, dx, \qquad \forall u \in H_{Neu},
\end{array}
$$
can be estimated as
$$
C^* \geq  c\, \delta^N \,  e^{c \,  \delta^{-N}} \min \Big\{ m_J  ; m_G \Big\}.
$$
for $\delta$ small enough
(here $c$ is a constant independent of $\delta$).
\end{remark}

Now we are ready to proceed with the proof of the existence and uniqueness
of minimizers of $E_{Neu}(u)$ in $H_{Neu}$.

\begin{theorem} \label{teo.1}
Given $f \in L^2 (\Omega)$ with $\int_\Omega f =0$ there exists a unique minimizer of $E_{Neu}(u)$ in $H_{Neu}$.
\end{theorem}

\begin{proof}[Proof of Theorem \ref{teo.1}] From the previous Poincar\'e-Wirtinger type inequality we obtain 
$$
\begin{array}{l}
\displaystyle 
E_{Neu} (u) = \int_{\Omega_\ell} \frac{|\nabla u(x)|^2}{2} \, dx 
+ \frac{1}{2}\int_{\Omega_{n\ell}}\int_{\Omega} J(x-y)|u(y)-u(x)|^2\, dy \,  dx
\\[10pt]
\qquad \qquad \qquad  \displaystyle  +
\frac{1}{2}\int_{\Omega_{n\ell}}\int_{\Omega_\ell} G(x,y)|u(y)-u(x)|^2\, dy\,  dx
- \int_\Omega f (x) u (x) \,dx \\[10pt]
\qquad \qquad \displaystyle  \geq c \int_{\Omega} u^2 (x) \,dx - \int_\Omega f(x)u(x) \,dx
\end{array}
$$
from where it follows that $E_{Neu} (u)$ is bounded below and coercive in $H_{Neu}$. Hence, existence of a minimizer follows
by the direct method of calculus of variations. Just take a minimizing sequence $u_n$ and extract a subsequence
that converges weakly in $L^2 (\Omega)$. Then, we have 
$$
0=\lim_{n\to \infty} \int_\Omega u_n (x) \,dx = \int_\Omega u(x) \,dx,
$$ 
$$
\lim_{n\to \infty} \int_\Omega f  (x)u_n (x) \,dx = \int_\Omega f(x)u(x) \,dx
$$ 
and 
$$
\begin{array}{l}
\displaystyle 
\lim_{n\to \infty}\int_{\Omega_\ell} \frac{|\nabla u(x)|^2}{2} \, dx 
+ \frac{1}{2}\int_{\Omega_{n\ell}}\int_{\Omega} J(x-y)|u(y)-u(x)|^2\, dy \,  dx
\\[10pt]
\qquad \qquad \qquad  \displaystyle  +
\frac{1}{2}\int_{\Omega_{n\ell}}\int_{\Omega_\ell} G(x,y)|u(y)-u(x)|^2\, dy\,  dx
\\[10pt]
\qquad \displaystyle  \geq \int_{\Omega_\ell} |\nabla u(x)|^2 \,dx 
+ \frac{\alpha}{2 }\int_{\Omega_{n\ell}}\int_{\Omega_{n\ell}} J(x-y)(u(y)-u(x))^2\, dy\, dx
\\[10pt]
\qquad \qquad \qquad  \displaystyle  +
\frac{1}{2}\int_{\Omega_{n\ell}}\int_{\Omega_\ell} G(x,y)|u(y)-u(x)|^2\, dy\,  dx.
\end{array}
$$
Then, we obtain that $u\in H_{Neu}$ and that $u$ is a minimizer,
$$E_{Neu}(u) = \min_{v\in H} E_{Neu}(v).$$ 

Uniqueness of minimizers follows from the strict convexity of the functional $E_{Neu} (u)$. 
\end{proof}

Associated with this energy we have an equation in $\Omega$.

\begin{lemma} \label{ec-dif.1}
The minimizer of $E_{Neu}(u)$ in $H_{Neu}$ is a weak solution to the following problem: a local
equation with a source in $\Omega_{\ell}$,
\begin{equation}
\label{eq:main.Neumann.local}
\begin{cases}
\displaystyle - f(x)=\Delta u (x) + \int_{\Omega_{n\ell}} G(x,y)(u(y)-u(x))\, dy,\ &x\in \Omega_\ell,\ t>0,
\\[10pt]
\displaystyle \frac{\partial u}{\partial \eta} (x)=0,\qquad & x\in \partial \Omega_\ell, ,
\end{cases}
\end{equation}
and a nonlocal equation with a source in $\Omega_{n\ell}$,
\begin{equation}
\label{eq:main.Neumann.nonlocal}
\displaystyle - f(x) =  2 \int_{\Omega_{n\ell}}\!\! J(x-y)(u(y)-u(x))\, dy
+ \int_{\Omega_\ell} G(x,y)(u(y)-u(x))\, dy,
\quad x \in \Omega_{n\ell}.
\end{equation}
\end{lemma}

\begin{proof} Let $u$ be the minimizer of $E_{Neu}(u)$ in $H_{Neu}$, then for every smooth $\varphi$
with $\int_{\Omega} \varphi =0$ and every $t\in \mathbb{R}$ we have
$$
E_{Neu}(u+t \varphi) - E_{Neu}(u) \geq 0.
$$
Therefore, we have that 
$$
\frac{\partial }{\partial t} E_{Neu}(u+t\varphi) |_{t=0} =0,
$$
that is,
$$
\begin{array}{l}
\displaystyle \int_\Omega f \varphi = \int_{\Omega_\ell} \nabla u \nabla \varphi +
 \int_{\Omega_{n\ell}}\int_{\Omega_{n\ell}} J(x-y)(u(y)-u(x)) (\varphi (y) - \varphi(x)) \, dy \, dx
\\[10pt]
\displaystyle \qquad \qquad  
+ \int_{\Omega_{n\ell}}\int_{\Omega_\ell} G(x,y)(u(y)-u(x))(\varphi(y) - \varphi(x)) \, dy\,  dx.
\end{array}
$$
Using that the kernel $J$ is symmetric and Fubini's Theorem we obtain
$$
\begin{array}{l}
\displaystyle \int_\Omega f \varphi = \int_{\Omega_\ell} \nabla u \nabla \varphi - 2
 \int_{\Omega_{n\ell}}\int_{\Omega_{n\ell}} J(x-y)(u(y)-u(x))  \, dy \, \varphi (x) \, dx \\[10pt]
\displaystyle \qquad \qquad  + \int_{\Omega_{n\ell}}\int_{\Omega_{\ell}} G(x,y)(u(y)-u(x))  \, dx \, \varphi (y)\, dy
\\[10pt]
\displaystyle \qquad \qquad  - \int_{\Omega_{n\ell}}\int_{\Omega_{\ell}} G(x,y)(u(y)-u(x))  \, dy \, \varphi (x)\, dx.
\end{array}
$$
From where if follows that $u$ is a weak solution to \eqref{eq:main.Neumann.local} and 
\eqref{eq:main.Neumann.nonlocal}.
\end{proof}

{\bf Connections with probability theory.} 
A probabilistic interpretation of this model in terms of particle systems runs as follows:
take an exponential clock that controls the jumps of the particles, 
in the local region $\Omega_{\ell}$ the particles move according to Brownian motion (with a reflexion
on the boundary) and when the clock rings a new position is sorted according to the kernel 
$G(x,\cdot)$, then they jump if the new position is in the nonlocal region $\Omega_{n\ell}$ (if the
sorted position lies in $\Omega_\ell$ then just continue moving by Brownian motion ignoring the
ring of the clock) while in the nonlocal region $\Omega_{n\ell}$ the particles stay still until the clock rings
and then they jump using the kernel $J(x-\cdot)$ or the kernel 
$G(x,\cdot)$ to select the new position in the whole $\Omega$.
Notice that the total number of particles inside the domain $\Omega$ remains constant in time. 

The minimizer to our functional $E_{Neu}(u)$ gives the stationary distribution of particles provided
that there is an external source $f$ (that adds particles where $f>0$ and remove particles where $f<0$).

\section{Mixed coupling} \label{sect-surface}

{\bf Second model. Coupling local/nonlocal problems via flux terms.}
Our aim now is to look for a scalar problem with an energy that combines local and nonlocal terms 
acting in different subdomains, $\Omega_\ell$ and $\Omega_{n\ell}$, of $\Omega$, but now the coupling
is made balancing the fluxes across a prescribed hypersurface, $\Gamma$.

Recall from the Introduction that we look for a minimizer of the energy
\begin{equation} \label{energy.909.888}
\begin{array}{l}
    F_{Neu} (u)  \displaystyle:=\frac{1}{2}\int_{\Omega_{l}} | \nabla u (x)|^2 \, dx + 
    \frac{1}{2}\int_{\Omega_{nl}}\int_{\Omega_{nl}}J(x-y)\left(u(y)-u(x)\right)^2 \, dy\, dx \\[10pt]
    \displaystyle \quad \qquad \qquad\qquad \qquad + \frac{1}{2}\int_{\Omega_{nl}}\int_{\Gamma}G(x,z)\left(u(x)-u(z)\right)^2 \, d\sigma(z) \, dx
    - \int_\Omega f(x) u(x) \, dx
\end{array}
\end{equation}
 in
$$
H_{Neu} = \left\{ u \in L^2 (\Omega) \ : \ u \in H^1 (\Omega_\ell) , \ \int_\Omega u(x) \, dx =0 \right\},
$$
assuming that 
$$
\int_\Omega f (x) \, dx =0.
$$

As before we first prove a lemma that gives coercivity for  our functional.
     	
	\begin{lemma} \label{lema.1.22}
		There exists a constant $C$ such that
		$$
		\begin{array}{l}
		\displaystyle \int_{\Omega_\ell} \frac{|\nabla u (x)|^2}{2} \,dx 
		+ \frac{1}{2}\int_{\Omega_{n\ell}}\int_{\Omega_{n\ell}} J(x-y)(u(y)-u(x))^2\, dy\, dx\\[10pt]
		\displaystyle \qquad \qquad \qquad \qquad \qquad 
		+ \frac{1}{2}\int_{\Omega_{nl}}\int_{\Gamma}G(x,z)\left(u(x)-u(z)\right)^2 d\sigma(z) dx
		 \geq C \int_\Omega |u (x)|^2 \,dx,
		\end{array}
		$$
		for every $u$ in $H_{Neu}$.
	\end{lemma}

	\begin{proof} We argue by contradiction, then we assume that there is a sequence $u_n\in H_{Neu}$ such that
		$$
		\int_\Omega |u_n (x)|^2 \, dx =1
		$$
		and 
		$$
		\begin{array}{l}
\displaystyle \int_{\Omega_\ell} \frac{|\nabla u_n (x)|^2}{2} \, dx 
+ \frac{1}{2}\int_{\Omega_{n\ell}}\int_{\Omega_{n\ell}} J(x-y)(u_n(y)-u_n(x))^2\, dy \, dx \\[10pt]
\qquad \displaystyle 
+ \frac{1}{2}\int_{\Omega_{nl}}\int_{\Gamma}G(x,z)\left(u_n(x)-u_n(z)\right)^2 \, d\sigma(z)\, dx
		\to 0.
		\end{array}
		$$
		Then we have that
		$$
		\int_{\Omega_\ell} |\nabla u_n(x)|^2 \, dx \to 0, 
		$$
		$$
	     \int_{\Omega_{nl}}\int_{\Gamma}G(x,z)\left(u_n(x)-u_n(z)\right)^2 \, d\sigma(z)\, dx \to 0,
		$$
		and
		$$ 
       \int_{\Omega_{n\ell}}\int_{\Omega_{n\ell}} J(x-y)(u_n(y)-v_n(x))^2\, dy\, dx
		\to 0.
		$$
		
		Since 
		$$
		\int_\Omega |u_n(x)|^2 \, dx =1
		$$
 we obtain that $u_n$ is bounded in $L^2(\Omega_\ell)$ and then, using that, 
		$$
		\int_{\Omega_\ell} \frac{|\nabla u_n(x)|^2}{2} \, dx \to 0 
		$$
		we get that there exists a constant $k_1$, such that, along a subsequence, 
		$$
		u_n \to k_1
		$$
		strongly in $H^1(\Omega_\ell)$. Hence, using the trace theorem on $\Gamma$,
		$H^1 (\Omega_\ell) \hookrightarrow  L^2 (\Gamma)$,  
		we obtain
		$$
		u_n \to k_1
		$$
		strongly $L^2(\Gamma)$.

		Now we argue in the nonlocal part $\Omega_{n\ell}$. Since $v_n$ is bounded in $L^2 (\Omega_{n\ell})$ and 
		$$ 
		\int_{\Omega_{n\ell}}\int_{\Omega_{n\ell}} J(x-y)(u_n(y)-u_n(x))^2\,  dy\, dx \to 0
		$$
		we have that (extracting another subsequence if necessary)
		$$
		u_n \rightharpoonup u
		$$
		weakly in $L^2(\Omega_{n\ell})$ and therefore, 
		$$
		\int_{\Omega_{n\ell}}\int_{\Omega_{n\ell}} J(x-y)(u(y)-u(x))^2\,  dy\, dx 
		\leq \lim_{n \to \infty}\frac{1}{2}\int_{\Omega_{n\ell}}\int_{\Omega_{n\ell}} J(x-y)(u_n(y)-u_n(x))^2\,  dy\, dx 
		=0.
		$$
		Hence, we get that $$u \equiv k_2$$ in $\Omega_{n\ell}$ (here we need 
		to assume that $\Omega_{n\ell}$ is $\delta-$connected).
	
From the weak convergence of $u_n$ to $u$ in
 $L^2(\Omega_{n\ell})$ and the strong convergence of $u_n$ to $u$ in
 $L^2(\Gamma)$ we obtain
$$
\int_{\Omega_{nl}}\int_{\Gamma}G(x,z)\left(k_2-k_1\right)^2 \, d\sigma(z)\, dx \leq \lim_{n \to \infty} 
\int_{\Omega_{nl}}\int_{\Gamma}G(x,z)\left(u_n(x)-u_n(z)\right)^2 \, d\sigma(z) \, dx =0.
$$
		
Hence, from the fact that
$$
\int_{\Omega_{nl}}\int_{\Gamma}G(x,z) \, d\sigma(z)\, dx >0,
$$ 
we obtain 
$$
k_1 = k_2.
$$
		
		Now, from the fact that $u_n \in H_{Neu}$, we have that
		$$
\int_{\Omega} u_n (x) \, dx = 0, 
		$$
		and then, passing to the limit, we obtain that
		$$
\int_{\Omega_{\ell}} k_1 + \int_{\Omega_{n\ell}} k_2 = 0.
		$$
		Hence, as $k_1=k_2$ we conclude that
		$$
		k_1 = k_2 = 0.
		$$
		
		Up to now we have that
		$u_n \to 0 $ strongly in $H^1(\Omega_{\ell})$ and $u_n \to 0$ weakly in $L^2(\Omega_{n\ell})$.
		Then, as
		$$
		1= \int_\Omega |u_n(x)|^2 \, dx = \int_{\Omega_{\ell}} |u_n(x)|^2 \, dx + \int_{\Omega_{n\ell}} |u_n(x)|^2 \, dx
		$$
		we get that
		$$
		\int_{\Omega_{n\ell}} |u_n(x)|^2 \, dx \to 1.
		$$
		Now, we go back to 
		$$
\int_{\Omega_{nl}}\int_{\Gamma}G(x,z)\left(u_n(x)-u_n(z)\right)^2 \, d\sigma(z)\, dx 
		\to 0
		$$
		and we obtain
		$$
		\begin{array}{l}
		\displaystyle 0 = \lim_{n \to \infty}
		\int_{\Omega_{nl}}\int_{\Gamma}G(x,z)\left(u_n(x)-u_n(z)\right)^2 \, d\sigma(z) \, dx \\[10pt]
		\qquad \displaystyle
		= \lim_{n \to \infty}
		\int_{\Omega_{nl}}\int_{\Gamma}G(x,z)|u_n(x)|^2 \, d\sigma(z)\,  dx
		\\[10pt]
		\qquad \displaystyle
		\qquad +  \lim_{n \to \infty} \int_{\Omega_{nl}}\int_{\Gamma}G(x,z)|u_n(z)|^2 \, d\sigma(z)\, dx
		\\[10pt]
		\qquad \displaystyle
		\qquad -  \lim_{n \to \infty} 2 \int_{\Omega_{nl}}\int_{\Gamma}G(x,z)u_n(x)u_n(z) \, d\sigma(z)\, dx.
		\end{array}
		$$
		Since $u_n \to 0$ strongly in $L^2(\Gamma)$, we have that 
		$$
		\lim_{n \to \infty}
		 \int_{\Omega_{nl}}\int_{\Gamma}G(x,z)|u_n(z)|^2 \, d\sigma(z) \, dx = 0
		$$
		and 
		$$
		\lim_{n \to \infty} \int_{\Omega_{nl}}\int_{\Gamma}G(x,z)u_n(x)u_n(z) \, d\sigma(z)\,  dx = 0.
		$$
		Therefore, we obtain that 
		\begin{equation}\label{acosta1.22}
		\lim_{n \to \infty} \int_{\Omega_{n\ell}} |u_n(x)|^2 \int_{\Gamma} G(x,z)  \,  dy\, dx = 0.
		\end{equation}
		
		On the other hand we have
		$$
		\begin{array}{l}
		\displaystyle 0 = \lim_{n \to \infty}
		\int_{\Omega_{n\ell}}\int_{\Omega_{n\ell}} J(x-y)(u_n(y)-u_n(x))^2\,  dy\, dx \\[10pt]
		\qquad \displaystyle
		= \lim_{n \to \infty} 2 \int_{\Omega_{n\ell}}\int_{\Omega_{n\ell}} J(x-y) |u_n(x)|^2 \,  dy\, dx
		\\[10pt]
		\qquad \displaystyle
		\qquad -  \lim_{n \to \infty} 2 \int_{\Omega_{n\ell}}\int_{\Omega_{n\ell}} J(x-y) u_n(y) u_n(x) \,  dy\, dx.
		\end{array}
		$$
		Now we use that $v_n \rightharpoonup 0$ weakly in $L^2 (\Omega_{n\ell})$ and that $J$ is continuous
		to obtain 
		$$
		\int_{\Omega_{n\ell}}\int_{\Omega_{n\ell}} J(x-y) u_n(y) dy \to 0
		$$
		strongly in $L^2(\Omega_{n\ell})$ and then we obtain 
		$$
		-  \lim_{n \to \infty} \int_{\Omega_{n\ell}}\int_{\Omega_{n\ell}} J(x-y) u_n(y) u_n(x) \, dy\, dx = 0.
		$$
		Hence, we get that
		\begin{equation}\label{acosta2.22}
		\lim_{n \to \infty} \int_{\Omega_{n\ell}} |v_n(x)|^2 \int_{\Omega_{n\ell}} J(x-y)  \, dy\, dx = 0.
		\end{equation}
		
		From \eqref{acosta1.22} and \eqref{acosta2.22} and the fact that there exists a constant $c$ such that
		$$
		\Big( \int_{\Omega_{n\ell}} J(x-y)  \, dy 
		+ \int_{\Gamma} G(x,z) d\sigma(z) \Big) \geq c >0,
		$$
		we conclude that
		$$
		\begin{array}{l}
		\displaystyle 
		0< c= c \lim_{n\to \infty } \int_{\Omega_{n\ell}} |u_n(x)|^2 dx \\[10pt]
		\displaystyle \qquad 
		\leq \lim_{n \to \infty} \int_{\Omega_{n\ell}} |u_n(x)|^2 \Big( \int_{\Omega_{n\ell}} J(x-y)  \,  dy 
		+ \int_{\Gamma} G(x,z) d\sigma(z) \Big) dx  
		\\[10pt]
		\displaystyle \qquad 
		= \lim_{n \to \infty} \int_{\Omega_{n\ell}} |u_n(x)|^2 \int_{\Omega_{n\ell}} J(x-y)  \,  dy\, dx 
		\\[10pt]
		\displaystyle \qquad \qquad 
		+  \lim_{n \to \infty} \int_{\Omega_{n\ell}} |u_n(x)|^2 \int_{\Gamma} G(x,z) d\sigma(z)  dx = 0
		\end{array}
		$$
		a contradiction. 
	\end{proof}
Similarly to Lemma \ref{lema.Poincare}, Lemma \ref{lema.1.22} does not provide information regarding the constant $c$. 
\begin{lemma}\label{lema.Poincare.con.cte.gamma} 
The optimal 
constant $c>0$ such that 
$$
		\begin{array}{l}
		\displaystyle \int_{\Omega_\ell} \frac{|\nabla u (x)|^2}{2} \,dx 
		+ \frac{1}{2}\int_{\Omega_{n\ell}}\int_{\Omega_{n\ell}} J(x-y)(u(y)-u(x))^2\, dy\, dx\\[10pt]
		\displaystyle \qquad \qquad \qquad \qquad \qquad 
		+ \frac{1}{2}\int_{\Omega_{nl}}\int_{\Gamma}G(x,z)\left(u(x)-u(z)\right)^2 d\sigma(z) dx
		 \geq c \int_\Omega |u (x)|^2 \,dx,
		\end{array}
		$$
for every function $u\in H_{Neu}$ (that is, for $u$ in the natural space with $\int_\Omega u =0$) can be estimated as
$$
c\geq c^* (J, G, \Omega_{\ell},\mathcal{T},\delta,\Gamma_A,B) 
$$
with 
$\Gamma_A\subset \Omega_{\ell}$, $B\subset  \Omega_{n\ell}$ such that $diam(\Gamma_A\cup B)<2\delta$ and $\mathcal{T}$ a $\delta-tree$ structure of $ \Omega_{n\ell}$ with root at $B$.
\end{lemma}

\begin{proof} The proof uses the same ideas as the ones of Lemma \ref{lema.Poincare.con.cte}. The only significant modifications are the following:  First, from our hypothesis on $G$ and the assumption $\mbox{dist}
(\Gamma_A, \Omega_{n\ell}) <\delta$, there exist two sets 
$\Gamma_A\subset  \Gamma$ and $B \subset  \Omega_{n\ell}$ such that 
$$
|x-y| \leq 2\delta, \qquad \forall x\in \Gamma_A, \,\forall y\in B.
$$
Take $r \in \mathbb{R}$ such that 
$$
\int_{\Gamma_A} (u(z) -r) \, d\sigma(z) =0. 
$$
Then, we have, for $v= u-r$ ­and calling  $|\Gamma_A|$ the $N-1$ dimensional measure of $\Gamma_A$ the following variant of \eqref{eq:intersec},
\begin{equation}
\label{eq:intersec2}
\begin{array}{l}
\displaystyle 
\frac{1}{2}\int_{\Omega_{nl}}\int_{\Gamma}G(x,z)\left(u(x)-u(z)\right)^2 d\sigma(z) dx
 \\[10pt]
\displaystyle \geq \frac{1}{2}\int_{B}\int_{\Gamma_A}G(x,z)\left(u(x)-u(z)\right)^2 d\sigma(z) dx
\\[10pt]
\displaystyle \geq m_G \frac{|\Gamma_A|}{2}\int_{B}|v(x)|^2\,  dx 
\end{array}
\end{equation}
where we are using that 
$$
\int_{B}|v(x)|^2\,  dx=\frac{1}{|\Gamma_A|^2}\int_{B}\left|\int_{\Gamma_A}(v(x)-v(z))\,d\sigma(z)\right|^2\,  dx \le
\frac{1}{|A|}\int_{B}\int_{\Gamma_A}|v(x)-v(z)|^2\,d\sigma(z)\,  dx.
$$
Using now that
 $$
0<\sigma (\Omega_\ell, \Gamma_A) = \inf_{v \in H^1 (\Omega_\ell), \int_{\Gamma_A} v \, d\sigma(z) =0} 
\frac{\displaystyle \int_{\Omega_\ell} |\nabla v|^2 (x) \, dx }{\displaystyle  \int_{\Omega_\ell} v^2 (x) \, dx } ,
$$
for a suitable Poincar\'e constant $\sigma (\Omega_\ell, \Gamma_A) $, the proof follows exactly as in Lemma \ref{lema.Poincare.con.cte}. 
\end{proof}
As before, we are ready to obtain existence and uniqueness of a minimizer of $F_{Neu}(u)$ in $H_{Neu}$.
	
	\begin{theorem} \label{teo.2.22}
Given $f \in L^2 (\Omega)$ with $\int_{\Omega} f =0$ there exists a unique minimizer of $F_{Neu}(u)$ in $H_{Neu}$.
\end{theorem}

\begin{proof}[Proof of Theorem \ref{teo.2.22}]
The previous lemma implies that $F_{Neu}(u)$ is bounded below and coercive. Hence, existence of a minimizer follows
just considering a minimizing sequence $(u_n)$ and extracting a subsequence
that converges weakly in $H^1(\Omega_{\ell}) \cap L^2 (\Omega_{n\ell})$. 
Uniqueness of a minimizer follows from the strict convexity of the functional $F_{Neu}(u)$. 
\end{proof}

	Now, we find the equations verified by the minimizer.

\begin{lemma} \label{lema.eq.Mod2.Neu}
The minimizer of $F_{Neu}$ is a weak solution to 
 \begin{equation}\label{local-N}
 \left\{
\begin{array}{ll}
 \displaystyle  f(x)  =  \Delta u (x),  \qquad  & x \in \Omega_\ell , \\[10pt]
 \displaystyle   \frac{\partial u}{\partial \eta}(x)  =  0, \qquad & x\in \partial \Omega_l \cap \partial \Omega,  \\[10pt]
 \displaystyle    \frac{\partial u}{\partial \eta}(x)  = \int_{\Omega_{nl}} G(y,x)( u(y)-  u(x)) dy, \qquad & 
 x \in \Gamma, 
    \end{array}
    \right.
    \end{equation}
    and
\begin{equation}\label{nonlocal-N} 
f(x)  = 2 \int_{\Omega_{nl}} J(x-y)\left(u(y)-u(x) \right)dy 
   -  \int_\Gamma G(x,z) ( u(x)-  u(z)) d\sigma (z) ,  \qquad x \in \Omega_{n\ell}. 
     \end{equation} 
     \end{lemma}
     
\begin{proof}
Let $u$ be the minimizer of $F_{Neu}(u)$ in $H_{Neu}$, then for every smooth $\varphi$
with $\int_{\Omega} \varphi =0$ and every $t\in \mathbb{R}$ we have
$$
F_{Neu}(u+t \varphi) - F_{Neu}(u) \geq 0.
$$
Therefore, we get that 
$$
\frac{\partial }{\partial t} F_{Neu}(u+t\varphi) |_{t=0} =0,
$$
that is,
$$
\begin{array}{l}
\displaystyle \int_\Omega f \varphi = \int_{\Omega_\ell} \nabla u \nabla \varphi +
 \int_{\Omega_{n\ell}}\int_{\Omega_{n\ell}} J(x-y)(u(y)-u(x)) (\varphi (y) -\varphi(x)) \, dy\, dx
 \\[10pt]
\displaystyle \qquad \qquad  + \int_{\Omega_{n\ell}}\int_{\Gamma} G(x,y)(u(y)-u(x)) \varphi (y) \,d \sigma(y) \, dx
\\[10pt]
\displaystyle \qquad \qquad  - \int_{\Omega_{n\ell}}\int_{\Gamma} G(x,y)(u(y)-u(x)) \varphi (x) \,d \sigma(y) \, dx.
\end{array}
$$
Using that the kernel $J$ is symmetric we obtain
$$
\begin{array}{l}
\displaystyle \int_\Omega f \varphi = \int_{\Omega_\ell} \nabla u \nabla \varphi - 2
 \int_{\Omega_{n\ell}}\int_{\Omega_{n\ell}} J(x-y)(u(y)-u(x))  \, dy\, \varphi(x) \, dx
 \\[10pt]
\displaystyle \qquad \qquad  - \int_{\Gamma} \int_{\Omega_{n\ell}} G(y,x)(u(y)-u(x)) \, dy \varphi (x) \, d \sigma(x)
\\[10pt]
\displaystyle \qquad \qquad  - \int_{\Omega_{n\ell}}\int_{\Gamma} G(x,y)(u(y)-u(x))  \, d \sigma(y)\, \varphi (x)\, dx,
\end{array}
$$
that is, $u$ is a weak solution to \eqref{local-N} and 
\eqref{nonlocal-N}.
\end{proof}

{\bf Probability interpretation.} A probabilistic interpretation of this model in terms of particle systems runs as follows:
in the local region $\Omega_{\ell}$ the particles move according to Brownian motion (with a reflexion
on the boundary of $\Omega_\ell$, $\partial \Omega_{\ell}$) 
and when they arrive to the surface $\Gamma$ they pass to the nonlocal region to a point 
selected using the kernel $G$; in the nonlocal region
take an exponential clock that controls the jumps of the particles, 
and when the clock rings a new position is sorted according to the kernel 
$J(x-\cdot)$ (for the movements inside $\Omega_{n\ell}$) or according to $G(x,\cdot)$ 
for jumping back to the local region entering at a point in the surface, $\Gamma$.
The total number of particles inside the domain $\Omega$ remains constant in time. 

The minimizer to our functional $F_{Neu}(u)$ gives the stationary distribution of particles ($u$ for the particles
inside the local region $\Omega_{\ell}$ and $v$ for particles in the nonlocal part $\Omega_{n\ell}$) provided
that there is an external source $f$ (that adds particles where $f>0$ and remove particles where $f<0$). 

Notice that here the coupling term appears as a flux boundary condition for the local part of the problem.

\section{Singular kernels} \label{sect-sing}

First, let us deal with
\begin{equation} \label{energy.sing.678}
\begin{array}{l}
\displaystyle 
E_{Neu}(u):=\int_{\Omega_\ell} \frac{|\nabla u (x)|^2}{2} \, dx 
+ \frac{1}{2}\int_{\Omega_{n\ell}}\int_{\Omega_{n\ell}}  \frac{C}{|x-y|^{N+2s}} |u(y)-u(x)|^2\, dy \, dx
\\[10pt]
\displaystyle \qquad\qquad  \qquad
+ \frac{1}{2}\int_{\Omega_{n\ell}}\int_{\Omega_\ell} G(x,y)|u(y)-u(x)|^2\, dy\,  dx
- \int_\Omega f (x) u (x) \, dx.
\end{array}
\end{equation}
We look for minimizers in the space
$$
H_{Neu} = \left\{ u: \, u|_{\Omega_\ell}  \in H^1 (\Omega_\ell),
u|_{\Omega_{n\ell}} \in H^s (\Omega_{n\ell}), \ \int_{\Omega} u =0 \right\}.
$$

In fact, now the key lemma (the Poincar\'e-Wirtinger type inequality) is simpler. 

\begin{lemma}\label{lema.Poincare.8888} There exists $c>0$ such that 
$$
\begin{array}{l}
\displaystyle
\int_{\Omega_\ell} |\nabla u|^2 + \frac{1}{2 }\int_{\Omega_{n\ell}}\int_{\Omega_{n\ell}}  \frac{C}{|x-y|^{N+2s}} 
(u(y)-u(x))^2\, dy\, dx
+ \frac{1}{2}\int_{\Omega_{n\ell}}\int_{\Omega_\ell} G(x,y)|u(y)-u(x)|^2\, dy\,  dx \\[10pt]
\qquad \qquad \qquad \displaystyle \geq \displaystyle c\int_{\Omega}u^2 (x) \, dx
\end{array}
$$
for every function $u\in H_{Neu}$.
\end{lemma}

\begin{proof} 
As before, we argue by contradiction. 
Assume that there is a sequence $u_n$ such that
\begin{equation} \label{eq.gradientes.99}
\int_{\Omega_\ell} |\nabla u_n|^2 (x) \, dx \to 0
\end{equation}
\begin{equation} \label{eq.J.99}
\frac{1}{2 }\int_{\Omega_{n\ell}}\int_{\Omega_{n\ell}}  \frac{C}{|x-y|^{N+2s}} |u_n(y)-u_n(x)|^2\, dy\, dx \to 0,
\end{equation}
\begin{equation} \label{eq.G.99}
\frac{1}{2}\int_{\Omega_{n\ell}}\int_{\Omega_\ell} G(x,y)|u_n(y)-u_n(x)|^2\, dy\,  dx \to 0,
\end{equation}
\begin{equation} \label{eq.=1.99}
\int_{\Omega}|u_n|^2 (x) \, dx = 1
\end{equation}
and 
\begin{equation} \label{eq.int-0.99}
\int_\Omega u_n (x) \, dx =0.
\end{equation}

Since the $L^2$ norm is bounded we can extract a subsequence such that
$$
u_n \rightharpoonup u \qquad \mbox{weakly in } H^1 (\Omega_{\ell}),  
$$
and 
$$
u_n \rightharpoonup u \qquad \mbox{weakly in } H^s (\Omega_{n\ell}).  
$$
then, from the compact embeddings $H^1 (\Omega_\ell)
\hookrightarrow L^2 (\Omega_\ell)$ and $H^s (\Omega_\ell)
\hookrightarrow L^2 (\Omega_\ell)$ and 
\eqref{eq.gradientes} and \eqref{eq.J.99} we get that there is are constants $k_1$, $k_2$ such that
$$
u_n \to k_1 \qquad \mbox{strongly in } L^2 (\Omega_{\ell})  
$$
and
$$
u_n \to k_2 \qquad \mbox{strongly in } L^2 (\Omega_{n\ell}).  
$$
From \eqref{eq.G} we conclude that 
$$
\frac{1}{2}\int_{\Omega_{n\ell}}\int_{\Omega_\ell} G(x,y)|k_1-k_2|^2\, dy\,  dx
\leq \lim_{n \to \infty}  \frac{1}{2}\int_{\Omega_{n\ell}}\int_{\Omega_\ell} G(x,y)|u_n(y)-u_n(x)|^2\, dy\,  dx =0,
$$
and hence
$$
k_1 = k_2. 
$$
Now, from \eqref{eq.int-0}, we obtain
$$
k_1=k_2=0. 
$$
Hence we get 
$$
\int_{\Omega}|u_n|^2 (x) \, dx
= \int_{\Omega_{\ell}}|u_n|^2 (x) \, dx+  \int_{\Omega_{n\ell}}|u_n|^2 (x) \, dx
\to 0.
$$
This contradicts that
$$
\int_{\Omega}|u_n|^2 (x) \, dx \to 1,
$$
and ends the proof.
\end{proof}

Now we are ready to proceed with the proof of Theorem \ref{teo.1.intro.sing}.

\begin{proof}[Proof of Theorem \ref{teo.1.intro.sing}] From the previous Poincar\'e-Wirtinger type inequality we obtain 
$$ 
E_{Neu} (u) \geq c \int_{\Omega} u^2 (x) \,dx - \int_\Omega f(x)u(x) \,dx
$$
from where it follows that $E_{Neu} (u)$ is bounded below and coercive in $H$. Hence, existence of a minimizer follows
by the direct method of calculus of variations. Just take a minimizing sequence $u_n$ and extract a subsequence
that converges weakly in $L^2 (\Omega)$. Then, using the lower semicontinuity of the seminorms
we obtain that $u\in H_{Neu}$ and that $u$ is a minimizer,
$$E_{Neu}(u) = \min_{v\in H} E_{Neu}(v).$$ 

Uniqueness of minimizers follows from the strict convexity of the functional $E_{Neu} (u)$. 

Associated with this energy we have an equation in $\Omega$ that can be obtained exactly as before.
\end{proof}

Now let us move to mixed couplings, 
\begin{equation} \label{energy.909.sing.678}
\begin{array}{l}
    F_{Neu} (u)  \displaystyle:=\frac{1}{2}\int_{\Omega_{l}} | \nabla u |^2 dx + 
    \frac{1}{2}\int_{\Omega_{nl}}\int_{\Omega_{nl}} \frac{C}{|x-y|^{N+2s}}
    \left(u(y)-u(x)\right)^2 dydx \\[10pt]
    \displaystyle \quad \qquad \qquad\qquad \qquad + \frac{1}{2}\int_{\Omega_{nl}}\int_{\Gamma}G(x,z)\left(u(x)-u(z)\right)^2 d\sigma(z) dx
    - \int_\Omega f (x) u (x) \,dx.
\end{array}
\end{equation}

The key lemma follows as before. 

\begin{lemma}\label{lema.Poincare.sing2} There exists $c>0$ such that 
$$
\begin{array}{l}
\displaystyle
\int_{\Omega_\ell} |\nabla u|^2 + \frac{1}{2 }\int_{\Omega_{n\ell}}\int_{\Omega_{n\ell}}  \frac{C}{|x-y|^{N+2s}} 
(u(y)-u(x))^2 dy dx
+ \frac{1}{2}\int_{\Omega_{n\ell}}\int_{\Gamma} G(x,y)|u(y)-u(x)|^2 d\sigma(y) dx \\[10pt]
\qquad \qquad \qquad \displaystyle \geq \displaystyle c\int_{\Omega}u^2 (x) \, dx
\end{array}
$$
for every function $u\in H_{Neu}$.
\end{lemma}

\begin{proof} 
The proof runs as before, arguing by contradiction. 
The only minor point is that, after proving that there is a subsequence $u_n$ such that 
$$
u_n \to k_1 \qquad \mbox{strongly in } L^2 (\Omega_{\ell})
$$
and
$$
u_n \to k_2 \qquad \mbox{strongly in } L^2 (\Omega_{n\ell}) 
$$
 we have to use that 
$$
\frac{1}{2}\int_{\Omega_{n\ell}}\int_{\Gamma} G(x,y)|k_1-k_2|^2\, d\sigma (y) \,  dx
\leq \lim_{n \to \infty}  \frac{1}{2}\int_{\Omega_{n\ell}}\int_{\Gamma} G(x,y)|u_n(y)-u_n(x)|^2\, d\sigma (y) \,  dx =0,
$$
to obtain
$$
k_1 = k_2. 
$$
The rest of the arguments until reaching a contradiction are exactly as before. 
\end{proof}

Now we are ready to proceed with the proof of Theorem \ref{teo.1}.

\begin{proof}[Proof of Theorem \ref{teo.1}] From the previous Poincar\'e-Wirtinger type inequality we obtain 
$$ 
E_{Neu} (u) \geq c \int_{\Omega} u^2 (x) \,dx - \int_\Omega f(x)u(x) \,dx
$$
from where it follows that $E_{Neu} (u)$ is bounded below and coercive in $H$. Hence, existence of a minimizer follows
by the direct method of calculus of variations. Just take a minimizing sequence $u_n$ and extract a subsequence
that converges weakly in $L^2 (\Omega)$. Then, using the lower semicontinuity of the seminorms
we obtain that $u\in H_{Neu}$ and that $u$ is a minimizer,
$$E_{Neu}(u) = \min_{v\in H} E_{Neu}(v).$$ 

Uniqueness of minimizers follows from the strict convexity of the functional $E_{Neu} (u)$. 

Associated with this energy we have an equation in $\Omega$ that can be obtained exactly as before.
\end{proof}

\medskip


\bibliographystyle{plain}

\end{document}